\documentclass[10pt,reqno]{amsart}
\RequirePackage[OT1]{fontenc}
\RequirePackage{amsthm,amsmath}
\RequirePackage[numbers]{natbib}
\RequirePackage[colorlinks,linkcolor=blue,citecolor=blue,urlcolor=blue]{hyperref}
\usepackage{amssymb}
\usepackage{anysize}
\usepackage{hyperref}
\usepackage{extarrows}
\usepackage{multirow}
\usepackage{natbib}
\usepackage{stmaryrd}
\usepackage{color}
\usepackage{amsmath}
\usepackage{enumitem}
\usepackage{mathrsfs}

\numberwithin{equation}{section}

\theoremstyle{plain} 
\newtheorem{theorem}{Theorem}[section]
\newtheorem{lemma}[theorem]{Lemma}
\newtheorem{corollary}[theorem]{Corollary}
\newtheorem{proposition}[theorem]{Proposition}

\newtheorem{definition}[theorem]{Definition}

\newcounter{kevin}
\numberwithin{kevin}{section}
\theoremstyle{remark}
\newtheorem{remark}[kevin]{Remark}

\renewcommand{\Re}{\mathrm{Re}\,}
\renewcommand{\Im}{\mathrm{Im}\,}

\newcommand{\im}{\mathrm{Im}\,}

\newcommand{\E}{{\mathbb E }}
\newcommand{\R}{{\mathbb R }}
\newcommand{\N}{{\mathbb N}}
\newcommand{\Z}{{\mathbb Z}}
\renewcommand{\P}{{\mathbb P}}
\newcommand{\C}{{\mathbb C}}

\newcommand{\ii}{\mathrm{i}}
\newcommand{\deq}{\mathrel{\mathop:}=}
\newcommand{\e}[1]{\mathrm{e}^{#1}}
\newcommand{\ntr}{\mathrm{tr}\,}
\newcommand{\dd}{\mathrm{d}}
\newcommand{\ie}{\emph{i.e., }}
\newcommand{\eg}{\emph{e.g., }}
\newcommand{\cf}{\emph{c.f., }}
\newcommand{\PP}{\Phi}
\newcommand{\dL}{\mathrm{d}_{\mathrm{L}}}
\newcommand{\wh}{\widehat}

\marginsize{35mm}{35mm}{38mm}{40mm}

\DeclareMathOperator*{\supp}{supp\,}

\DeclareMathOperator*{\IE}{ \mathbb{I}\mkern-2mu\mathbb{E}}

\begin{document}

 \begin{minipage}{0.85\textwidth}
 \vspace{2.5cm}
 \end{minipage}
\begin{center}
\large\bf
Local Stability of the Free Additive Convolution
\end{center}

\renewcommand{\thefootnote}{\fnsymbol{footnote}}	
\vspace{1cm}
\begin{center}
 \begin{minipage}{0.3\textwidth}
\begin{center}
Zhigang Bao\footnotemark[2]  \\
\footnotesize {IST Austria}\\
{\it zhigang.bao@ist.ac.at}
\end{center}
\end{minipage}
\begin{minipage}{0.3\textwidth}
\begin{center}
L\'aszl\'o Erd{\H o}s\footnotemark[1]  \\
\footnotesize {IST Austria}\\
{\it lerdos@ist.ac.at}
\end{center}
\end{minipage}
\begin{minipage}{0.3\textwidth}
 \begin{center}
Kevin Schnelli\footnotemark[2]\\
\footnotesize 
{IST Austria}\\
{\it kevin.schnelli@ist.ac.at}
\end{center}
\end{minipage}
\footnotetext[1]{Partially supported by ERC Advanced Grant RANMAT No.\ 338804.}
\footnotetext[2]{Supported by ERC Advanced Grant RANMAT No.\ 338804.}

\renewcommand{\thefootnote}{\fnsymbol{footnote}}	

\end{center}
\vspace{1cm}

\begin{center}
 \begin{minipage}{0.8\textwidth}\footnotesize{}
 We prove that the system of subordination equations, defining the free additive convolution of two probability measures,
is stable away from the edges of the support and blow-up singularities
by showing that the recent smoothness condition of Kargin is always satisfied.
As an application, we consider the local  spectral statistics 
 of the random matrix ensemble  $A+UBU^*$, where  $U$ is a Haar distributed random unitary 
 or orthogonal matrix, and $A$ and $B$ are deterministic matrices. In the bulk regime, we prove that the empirical spectral distribution of~$A+UBU^*$ concentrates around the free additive convolution of the spectral distributions of $A$ and $B$ on scales down to~$N^{-2/3}$.

\end{minipage}
\end{center}

 \vspace{2mm}
 
 {\small

 \footnotesize{\noindent\textit{Keywords}: Free convolution, subordination, local eigenvalue density}
 
 \footnotesize{\noindent\textit{AMS Subject Classification (2010)}: 46L54, 60B20}
 \vspace{2mm}

 }

\thispagestyle{headings}

\section{Introduction}

One of the basic concepts of free probability theory is the free additive convolution of two
probability laws in a non-commutative probability space; it describes the law of the sum of
two free random variables. In the case of a bounded self-adjoint random variable, its law can be identified 
with a probability measure of compact support on the real line. Hence the
free additive convolution of two probability measures is a well-defined concept and it is characteristically different from the classical convolution. 

In this paper, we prove a local stability result of the free additive convolution. A direct consequence
 is the continuity of the free additive convolution in a much stronger topology 
than established earlier by  Bercovici and Voiculescu~\cite{BeV93}.
A second application of our stability result is  to establish a local law on a very small scale for the eigenvalue
density of a random matrix ensemble $A+UBU^*$ where $U$ is a Haar distributed unitary or orthogonal matrix
and $A$, $B$ are deterministic $N$~by~$N$ hermitian matrices.

The free additive convolution was originally introduced by Voiculescu~\cite{Voi86} for the sum of free bounded 
noncommutative random variables in an algebraic setup (see  
 Maassen~\cite{Maa92} and by Bercovici and Voiculescu~\cite{BeV93}
for extensions to the unbounded case).
The Stieltjes transform of the free additive convolution is related 
to the Cauchy-Stieltjes transforms of the original measures by an elegant analytic change of variables.
 This {\it subordination phenomenon}
was first  observed by Voiculescu~\cite{Voi93} in a generic situation
and extended to full generality by Biane~\cite{Bia98}. In fact, the subordination equations,
see~\eqref{le definiting equations}-\eqref{le kkv} below, may directly be used to define
the free additive convolution. This analytic definition was given independently by Belinschi and Bercovici~\cite{BB} and by
Chistyakov and G\"{o}tze~\cite{CG}; for further details we refer to, \eg~\cite{VDN,HP,AGZ}.

Kargin~\cite{Kargin2013} pointed out 
that the analytic approach to the subordination equations, in contrast to the algebraic one,
allows one to effectively study how free additive convolution is affected by small perturbations;
this is especially useful to treat various error terms in the random matrix problem~\cite{Kargin}.
The basic tool is a local stability analysis of the subordination equations. In~\cite{Kargin2013}, Kargin assumed a lower bound on the imaginary part of the subordination functions
and a certain non-degeneracy condition on the Jacobian that holds for generic values of the spectral parameter.
While these so-called {\it smoothness conditions} hold 
in many examples, 
a general characterization was lacking.
Our first result, Theorem~\ref{thm stability}, shows that the smoothness
 conditions hold wherever the absolutely continuous part of
the free convolution measure is finite and  nonzero. In particular,  local stability holds unconditionally
(Corollary~\ref{corollary perturbation}) 
and, following  Kargin's argument~\cite{Kargin2013}, we immediately obtain the continuity of the free additive convolution in a stronger sense; see Theorem~\ref{le theorem continuity}.

The random matrix application of this stability result, however, goes well beyond Kargin's analysis~\cite{Kargin}
since our proof is valid on a much smaller scale. To explain the new elements, we recall how free probability connects to random matrices.

The following fundamental observation was made by Voiculescu~\cite{Voi91}
(later extended by Dykema~\cite{Dyk93b} and Speicher~\cite{Spe94}): if
$A=A^{(N)}$ and $B=B^{(N)}$ are two sequences of  Hermitian matrices that are asymptotically free with eigenvalue
distributions converging to probability measures $\mu_\alpha$ and $\mu_\beta$, then the eigenvalue density of $A+B$ is asymptotically given
by the free additive convolution $\mu_\alpha\boxplus\mu_\beta$.
 One of the most natural ways to
ensure asymptotic freeness is to consider conjugation by independent unitary matrices. Indeed,
if~$A$ and~$B$ are deterministic (may even be chosen diagonal)
with limit laws $\mu_\alpha$ and $\mu_\beta$, then $A$ and $UBU^*$ are asymptotically free if $U=U^{(N)}$ is a Haar distributed matrix; see~\cite{Voi91} and many subsequent works, \eg~\cite{Spe93,Xu97,Bia98bis,VP,Col03}. In particular, the limiting spectral density of the eigenvalues of $H=A+UBU^*$ is given by 
$\mu_\alpha\boxplus\mu_\beta$.

The conventional setup of free probability operates with moment calculations. An alternative approach~\cite{VP} proves the convergence of the resolvent at
any fixed spectral parameter~$z\in \C^+$. Both approaches give rise to weak convergence of measures, in particular they identify the limiting spectral density on macroscopic scale.

Armed with these macroscopic results,
 it is natural to ask for a {\it local law}, \ie for the smallest possible ($N$-dependent) scale so that the 
local eigenvalue density on that scale still converges as  $N$ tends to infinity. Local laws have been somewhat outside of the focus of free probability before Kargin's recent works.
 After having improved a concentration 
result for the Haar measure by Chatterjee~\cite{Chatterjee} by using the Gromov-Milman concentration inequality, Kargin obtained a
local law for the ensemble $H=A+UBU^*$  on scale  $\eta \gg (\log N)^{-1/2}$~\cite{Kargin2012}, \ie slightly below the macroscopic scale. Recently in~\cite{Kargin}, he improved this result down to scale $\eta\gg N^{-1/7}$
 under the above mentioned smoothness condition.
In Theorem~\ref{thm041801} we prove the local law 
down to scale $\eta = \im z\gg N^{-2/3}$ without any
additional assumption. 

To achieve this short scale, we effectively use the positivity of the imaginary
parts of the subordination functions by localizing the Gromov--Milman concentration
inequality within the spectrum.
 Since the subordination functions are obtained as 
the solution of a system of self-consistent equations whose derivation itself requires
bounds on the subordination functions, the reasoning seems circular. We break this
circularity by a continuity argument (similarly as in~\cite{ESY}) in which we reduce the imaginary part of
the spectral parameter in very small steps, use the previous step as an {\it a priori} bound
and show that the bound does not deteriorate by using the local  stability result, Theorem~\ref{le kkv}.

Finally, we remark that the local stability result is also a key ingredient in~\cite{BES15-2},
where we were able to prove a local law down to the smallest possible scale $\eta\gg N^{-1}$, but
with a weaker error bound than in Theorem~\ref{thm041801}; see Remark~\ref{remark:companion} for details.

\subsection{Notation}
 We use the symbols $O(\,\cdot\,)$ and $o(\,\cdot\,)$ for the standard big-O and little-o notation. We use~$c$ and~$C$ to denote positive numerical constants. Their values may change from line to line. For $a,b>0$, we write $a\lesssim b$, $a\gtrsim b$ if there is $C\ge1$ such that $a\le Cb$, $a\geq C^{-1} b$ respectively.  We write $a\sim b$, if $a\lesssim b$ and $a\gtrsim b$ both hold. We denote by $\|v\|$ the Euclidean norm of $v\in\C^N$. For an $N\times N$ matrix $A\in M_N(\C)$, we denote by $\|A\|$ its operator norm and by $\|A\|_2\deq\sqrt{\langle A,A\rangle}$ its Hilbert-Schmidt norm, where $\langle A,B \rangle\deq\mathrm{Trace}(AB^*)$, for $A,B\in M_N(\C)$. Finally, we denote by $\ntr\! A$ the normalized trace of $A$, \ie $\ntr\! A=\frac{1}{N}\mathrm{Trace}\,A$. 

\subsection*{Acknowledgment}
We thank an anonymous referee for many useful comments and remarks, and bringing references~\cite{Bel2,BW} to our attention.

\section{Main results} \label{section: Main results}

\subsection{Free additive convolution} \label{le subsection additive convolution} 
In this subsection, we recall the definition of the free additive convolution. Given a probability measure\footnote{All probability measures considered will be assumed to be Borel.} $\mu$ on $\R$, its {\it Stieltjes transform}, $m_\mu$, on the complex upper half-plane $\C^+\deq\{ z\in\C\,:\, \im z>0\}$ is defined by
\begin{align}\label{le definition of stieltjes transform}
 m_\mu(z)\deq\int_\R\frac{\dd\mu(x)}{x-z}\,, \qquad\qquad z\in\C^+\,.
\end{align}
We denote by $F_\mu$ the {\it negative reciprocal Stieltjes transform} of $\mu$, \ie
\begin{align}\label{le F definition}
 F_{\mu}(z)\deq -\frac{1}{m_{\mu}(z)}\,,\qquad \qquad z\in\C^+\,.
\end{align}
Observe that
 \begin{align}\label{le F behaviour at infinity}
\lim_{\eta\nearrow \infty}\frac{F_{\mu}(\ii\eta)}{\ii\eta}=1\,,
\end{align}
as follows easily from~\eqref{le definition of stieltjes transform}. Note, moreover, that $F_\mu$ is an analytic function on $\C^+$ with non-negative imaginary part. Conversely, if $F\,:\, \C^+\rightarrow \C^+$ is an analytic function such that $\lim_{\eta\nearrow\infty}  F(\ii\eta)/\ii\eta=1$, then $F$ is the negative reciprocal Stieltjes transform of a probability measure $\mu$, \ie $F(z)=F_\mu(z)$, for all $z\in\C^+$; see, \eg ~\cite{Aki}.

The {\it free additive convolution} is the binary operation on probability measures on $\R$ characterized by the following result.
\begin{proposition}[Theorem 4.1 in~\cite{BB}, Theorem~2.1 in~\cite{CG}]\label{le prop 1}
Given two probability measures $\mu_1$ and $\mu_2$ on $\R$, there exist unique analytic functions, $\omega_1\,,\omega_2\,:\,\C^+\rightarrow \C^+$, such that,
 \begin{itemize}[noitemsep,topsep=0pt,partopsep=0pt,parsep=0pt]
  \item[$(i)$] for all $z\in \C^+$, $\im \omega_1(z),\,\im \omega_2(z)\ge \im z$, and
  \begin{align}\label{le limit of omega}
  \lim_{\eta\nearrow\infty}\frac{\omega_1(\ii\eta)}{\ii\eta}=\lim_{\eta\nearrow\infty}\frac{\omega_2(\ii\eta)}{\ii\eta}=1\,;
  \end{align}
  \item[$(ii)$] for all $z\in\C^+$, 
\begin{align}
 \begin{aligned}\label{le definiting equations}
   &F_{\mu_1}(\omega_{2}(z))-\omega_1(z)-\omega_2(z)+z=0\,, \\
   &F_{\mu_2}(\omega_1(z))-\omega_1(z)-\omega_2(z)+z=0\,.
  \end{aligned}
\end{align}
 \end{itemize}
\end{proposition}
It follows from~\eqref{le limit of omega} that the analytic function $F\,:\,\C^+\rightarrow \C^+$ defined by
\begin{align}\label{le kkv}
 F(z)\deq F_{\mu_1}(\omega_{2}(z))=F_{\mu_2}(\omega_{1}(z))\,,
\end{align}
satisfies~\eqref{le F behaviour at infinity}. Thus $F$ is the negative reciprocal Stieltjes transform of a probability measure~$\mu$, called the free additive convolution of $\mu_1$ and $\mu_2$,  usually denoted by $\mu\equiv\mu_1\boxplus\mu_2$. Note that~\eqref{le kkv} shows that the r\^oles of $\mu_1$ and $\mu_2$ are symmetric and thus $\mu_1\boxplus\mu_2=\mu_2\boxplus\mu_1$. The functions $\omega_1$ and $\omega_2$ of Proposition~\ref{le prop 1} are called {\it subordination functions} and $F$ is said to be subordinated to~$F_{\mu_1}$, respectively to $F_{\mu_2}$. 

We mention that Voiculescu~\cite{Voi86} originally introduced the free additive convolution in a different, algebraic manner.
The equivalent analytic definition 
based on the existence of subordination functions (taken up in Proposition~\ref{le prop 1} above) was introduced in~\cite{BB, CG}.

We next recall some basic examples. Choosing $\mu_1$ arbitrary and $\mu_2$ as a single point mass at $b\in \R$, it is easy to check that $\mu_{1}\boxplus\mu_2 $ simply is $\mu_1$ shifted by $b$. We exclude this uninteresting case by henceforth assuming that $\mu_1$ and~$\mu_2$ are both supported at more than one point. Choosing $\mu_1=\mu_2=\mu$ as the Bernoulli distribution
\begin{align*}
 \mu=(1-\xi)\delta_{0}+\xi \delta_{1}\,,\qquad\qquad \xi\in(0,1)\,,
\end{align*}
the free additive convolution is explicitly given by (see \eg~(5.5) of~\cite{VP})
\begin{align}\label{prototype of two point mass}
 (\mu\boxplus\mu)(x)=\frac{\sqrt{(\ell_+-x)_+(x-\ell_-)_+}}{\pi x(2-x)}+(1-2\xi)_+\delta_{0}(x)+(2\xi-1)_+\delta_2(x)\,,
\end{align}
$x\in\R$, where $\ell_\pm\deq 1\pm 2\sqrt{\xi(1-\xi)}$ and where $(\,\cdot\,)_+$ denotes the positive part. Observe that $\mu\boxplus\mu$ has a nonzero absolutely continuous part and, depending on the choice of $\xi$, a point mass. Another important choice for $\mu_2$ is Wigner's semicircle law $\mu_{\mathrm{sc}}$. For arbitrary $\mu_1$, $\mu_1\boxplus\mu_{\mathrm{sc}}$ is then purely absolutely continuous with a bounded density\footnote{All densities are with respect to Lebesgue measure on $\R$.}  that is real analytic wherever positive~\cite{B}. 

Returning to the generic setting, the atoms of $\mu_1\boxplus\mu_2$ are identified as follows.  A point $c\in\R$ is an atom of $\mu_1\boxplus\mu_2$, if and only if there exist $a,b\in\R$ such that $c=a+b$ and $\mu_1(\{a\})+\mu_2(\{b\})>1$; see [Theorem~7.4,~\cite{BeV98}]. For another interesting properties of the atoms of $\mu_1\boxplus\mu_2$ we refer the reader to~\cite{BW}. The boundary behavior of the functions $F_{\mu_1\boxplus\mu_2}$,~$\omega_1$ and~$\omega_2$ has been studied by Belinschi~\cite{Bel1,Bel,Bel2} who proved the next two results. For simplicity, we restrict the discussion to compactly supported probability measures.
\begin{proposition}[Theorem 2.3 in~\cite{Bel1}, Theorem~3.3 in~\cite{Bel}]\label{prop extension}
 Let $\mu_1$ and $\mu_2$ be compactly supported probability measures on $\R$, none of them being a single point mass. Then the functions $F_{\mu_1\boxplus\mu_2}$, $\omega_1$, $\omega_2\,:\, \C^+\to\C^+$ extend continuously to $\R$.
\end{proposition}
Belinschi further showed in Theorem~4.1 in~\cite{Bel} that the singular continuous part of $\mu_1\boxplus\mu_2$ is always zero and that the absolutely continuous part, $(\mu_1\boxplus\mu_2)^{\mathrm{ac}}$, of $\mu_1\boxplus\mu_2$ is always nonzero. We denote the density function of $(\mu_1\boxplus\mu_2)^{\mathrm{ac}}$ by $f_{\mu_1\boxplus\mu_2}$.

We are now ready to introduce our notion of {\it regular bulk}, $\mathcal{B}_{\mu_1\boxplus\mu_2}$, of $\mu_1\boxplus\mu_2$. Informally, we let $\mathcal{B}_{\mu_1\boxplus\mu_2}$ be the open set on which $\mu_1\boxplus\mu_2$ admits a continuous density that is strictly positive and bounded from above. For a formal definition we first introduce the set
\begin{align}\label{le set U}
\mathcal{U}_{\mu_1\boxplus\mu_2}\deq\mathrm{int}\,\bigg\{\supp (\mu_1\boxplus\mu_2)^{\mathrm{ac}}\,\big\backslash\,\{ x\in \R\,:\, \lim_{\eta\searrow 0}F_{\mu_1\boxplus\mu_2}(x+\ii\eta)=0\} \bigg\}\,.
\end{align}
Note that $\mathcal{U}_{\mu_1\boxplus\mu_2}$ does not contain any atoms of $\mu_1\boxplus\mu_2$.  By Privalov's  theorem the set $ \{ x\in \R\,:\, \lim_{\eta\searrow 0}F_{\mu_1\boxplus\mu_2}(x+\ii\eta)=0\}$ has Lebesgue measure zero. In fact, an even stronger statement applies for the case at hand. Belinschi~\cite{Bel2} showed that if $x\in\R$ is such that $\lim_{\eta\searrow 0}F_{\mu_1\boxplus\mu_2}(x+\ii\eta)=0$, then it must be of the form $x=a+b$ with $\mu_1(\{a\})+\mu_2(\{b\})\ge 1$, $a,b\in\R$. There could only be  finitely many such $x$, thus $\mathcal{U}_{\mu_1\boxplus\mu_2}$ must contain an open non-empty interval.
\begin{proposition}[Theorem~3.3 in~\cite{Bel}]\label{le real analytic prop}
Let $\mu_1$ and $\mu_2$ be as above and fix any $x\in\mathcal{U}_{\mu_1\boxplus\mu_2}$. Then $F_{\mu_1\boxplus\mu_2}$, $\omega_1$, $\omega_2\,:\,\C^+\rightarrow \C^+$ extend analytically around $x$. In particular, the density function $f_{\mu_1\boxplus\mu_2}$ is real analytic in $\mathcal{U}_{\mu_1\boxplus\mu_2}$ wherever positive.
\end{proposition}
The regular bulk is obtained from $\mathcal{U}_{\mu_1\boxplus\mu_2}$ by removing the zeros of $f_{\mu_1\boxplus\mu_2}$ inside $\mathcal{U}_{\mu_1\boxplus\mu_2}$.
\begin{definition}
The regular bulk of the measure $\mu_1\boxplus\mu_2$ is defined as the set
\begin{align}
 \mathcal{B}_{\mu_1\boxplus\mu_2}\deq\mathcal{U}_{\mu_1\boxplus\mu_2}\,\backslash\, \left\{x\in\mathcal{U}_{\mu_1\boxplus\mu_2}\,:\,f_{\mu_1\boxplus\mu_2}	(x)=0\right\}\,.
\end{align}
\end{definition}
Note that $\mathcal{B}_{\mu_1\boxplus\mu_2}$ is an open non-empty set on which $\mu_1\boxplus\mu_2$ admits the density $f_{\mu_1\boxplus\mu_2}$. The density is strictly positive and thus (by Proposition~\ref{le real analytic prop}) real analytic on $\mathcal{B}_{\mu_1\boxplus\mu_2}$.

\subsection{Stability Result}
To present our results it is convenient to recast~\eqref{le definiting equations} in a compact form: For generic probability measures $\mu_1,\mu_2$ as above, let the function $\PP_{\mu_1,\mu_2}\,:\, (\C^+)^{3}\rightarrow \C^2$ be given by
\begin{align}\label{le H system defs}
\PP_{\mu_1,\mu_2}(\omega_1,\omega_2,z)\deq\left(\begin{array}{cc} F_{\mu_1}(\omega_2)-\omega_1-\omega_2+z \\ F_{\mu_2}(\omega_1)-\omega_1-\omega_2+z \end{array}\right)\,.
\end{align}
Considering $\mu_1,\mu_2$ as fixed, the equation
\begin{align}\label{le H system}
\PP_{\mu_1,\mu_2}(\omega_1,\omega_2,z)=0\,,
\end{align}
is equivalent to~\eqref{le definiting equations} and, by Proposition~\ref{le prop 1}, there are unique analytic functions $\omega_1,\omega_2\,:\, \C^+\rightarrow \C^+$, $z\mapsto \omega_1(z),\omega_2(z)$ satisfying~\eqref{le limit of omega} that solve~\eqref{le H system} in terms of $z$. We use the following conventions: We denote by $\omega_1$ and $\omega_2$ generic variables on $\C^+$ and we denote, with a slight abuse of notation, by $\omega_1(z)$ and $\omega_2(z)$ the subordination functions solving~\eqref{le H system} in terms of~$z$. When no confusion can arise, we simply write~$\PP$ for $\PP_{\mu_1,\mu_2}$. 

We call the system~\eqref{le H system} {\it linearly $S$-stable} at $(\omega_1,\omega_2)$ if
\begin{align}\label{le what stable means}
\Gamma_{\mu_1,\mu_2}(\omega_1,\omega_2)\deq \left\|\left(\begin{array}{cc}
-1& F_{\mu_1}'(\omega_2)-1  \\
F_{\mu_2}'(\omega_1)-1& -1\\
  \end{array}\right)^{-1} \right\|\le S\,,
\end{align}
for some constant $S$. Especially, the partial Jacobian matrix, $\mathrm{D}\PP(\omega_1,\omega_2)$, of~\eqref{le H system defs} given by
\begin{align*}
 \mathrm{D}\PP(\omega_1,\omega_2)\deq\left(\frac{\partial \PP}{\partial \omega_1}(\omega_1,\omega_2,z) \,,\,\frac{\partial \PP}{\partial\omega_2}(\omega_1,\omega_2,z) \right)=\left(\begin{array}{cc}
-1& F_{\mu_1}'(\omega_2)-1  \\
F_{\mu_2}'(\omega_1)-1 & -1\\
  \end{array}\right),
\end{align*}
admits a bounded inverse at $(\omega_1,\omega_2)$. Note that $ \mathrm{D}\PP(\omega_1,\omega_2)$ is independent of $z$.

Our first main result shows that the system~\eqref{le H system} is linearly stable and that the imaginary parts of the subordination functions are bounded below in the regular bulk. We require some more notation: For $a,b\ge 0$, $b\ge a$, and an interval $\mathcal{I}\subset \R$, we introduce the domain
\begin{align}\label{le domain S}
 \mathcal{S}_{\mathcal{I}}(a,b)\deq \{z=E+\ii\eta\in\C^+\,:\,E\in \mathcal{I}\,,  a\le \eta\le b\}\,.
\end{align}

\begin{theorem} \label{thm stability}
Let $\mu_1$ and $\mu_2$ be compactly supported probability measures on $\R$, and assume that neither is supported at a single point and that at least one of them is supported at more than two points. Let $\mathcal{I}\subset\mathcal{B}_{\mu_1\boxplus\mu_2}$ be a compact non-empty interval and fix some $0<\eta_{\mathrm{M}}<\infty$. 

Then there are two constants $k>0$ and $S<\infty$, both depending on the measures~$\mu_1$ and~$\mu_2$, on the interval $\mathcal{I}$ as well as on the constant $\eta_{\mathrm{M}}$, such that following statements hold.
\begin{itemize}[noitemsep,topsep=0pt,partopsep=0pt,parsep=0pt]
\item[$(i)$] The imaginary parts $\im \omega_1$ and $\im \omega_2$ of the subordination functions associated with~$\mu_1$ and~$\mu_2$ satisfy
\begin{align}\label{le thm lower bound on omega equation}
  \min_{z\in \mathcal{S}_{\mathcal{I}}(0,\eta_{\mathrm{M}})}\im \omega_1(z)\ge 2k\,,\qquad\qquad \min_{z\in \mathcal{S}_{\mathcal{I}}(0,\eta_{\mathrm{M}})}\im \omega_2(z)\ge 2k\,.
\end{align}
\item[$(ii)$] The system $\PP_{\mu_1,\mu_2}(\omega_1,\omega_2,z)=0$ is linearly $S$-stable at $(\omega_1(z),\omega_2(z))$ uniformly in $\mathcal{S}_{\mathcal{I}}(0,\eta_{\mathrm{M}})$, \ie
\begin{align}\label{le thm linear stability equation}
\max_{z\in\mathcal{S}_{\mathcal{I}}(0,\eta_{\mathrm{M}})}\Gamma_{\mu_1,\mu_2}(\omega_1(z),\omega_2(z))\le S\,.
\end{align}
\end{itemize}
\end{theorem}
\begin{remark}
 The assumption that neither of $\mu_1$, $\mu_2$ is a point mass guarantees that the free additive convolution is not a simple translate.  The case when both, $\mu_1$ and $\mu_2$ are combinations of two point masses is special and its discussion is postponed to Section~\ref{le final section}.
\end{remark}

Theorem~\ref{thm stability} has the following local stability result as corollary.
\begin{corollary}\label{corollary perturbation}
Let $\mu_1$, $\mu_2$ and $\mathcal{S}_{\mathcal{I}}(0,\eta_{\mathrm{M}})$ be as in Theorem~\ref{thm stability}. Fix $z_0\in\C^+$. Assume that the functions $\widetilde\omega_1$, $\widetilde\omega_2$, $\widetilde{r}_1$, $\widetilde{r}_2\,:\,\C^+\rightarrow \C$ satisfy $\im\widetilde\omega_1(z_0)>0$, $\im\widetilde\omega_2(z_0)>0$ and 
\begin{align}
 \PP_{\mu_1,\mu_2}(\widetilde\omega_1(z_0),\widetilde\omega_2(z_0),z_0)=\widetilde r(z_0)\,,
\end{align}
with $\widetilde r(z)\deq(\widetilde r_1(z),\widetilde r_2(z))^\top$. Let $\omega_1$, $\omega_2$ be the subordination functions solving the system $\PP_{\mu_1,\mu_2}(\omega_1(z),\omega_2(z),z)=0$, $z\in\C^+$.

Then there exists a (small) constant $\delta_0>0$ such that whenever we have
\begin{align}\label{le huhu}
 |\widetilde\omega_1(z_0)-\omega_1(z_0)|\le\delta_0\,,\qquad |\widetilde\omega_1(z_0)-\omega_1(z_0)|\le \delta_0\,,
\end{align}
we also have
\begin{align}
 |\widetilde\omega_1(z_0)-\omega_1(z_0)|\le 2S\|\widetilde r(z_0)\|\,,\qquad\qquad |\widetilde\omega_2(z_0)-\omega_2(z_0)|\le 2S\|\widetilde r(z_0)\|\,.
\end{align}
The constant $\delta_0>0$ depends on~$\mu_1$ and~$\mu_2$, on the interval $\mathcal{I}$ as well as on~$\eta_{\mathrm{M}}$.
\end{corollary}
We omit the proof of Corollary~\ref{corollary perturbation} from Theorem~\ref{thm stability}, since it follows directly from Proposition~\ref{le proposition perturbation of system} in Section~\ref{section perturbation} below.

\subsection{Applications}
We next explain two main applications of the stability estimates obtained in Theorem~\ref{thm stability}.
\subsubsection{Continuity of the free additive convolution}
Our first application shows that the free additive convolution is a continuous operation when the image is equipped with the topology of local uniform convergence of the density in the regular bulk; see~\eqref{le uniform pointwise}. Bercovici and Voiculescu (Proposition~4.13 of~\cite{BeV93}) showed that the free additive convolution is continuous with respect to weak convergence of measures. More precisely, given two pairs of probability measures $\mu_A$, $\mu_B$ and $\mu_\alpha$, $\mu_\beta$ on $\R$, the measures $\mu_A\boxplus\mu_B$ and $\mu_\alpha\boxplus\mu_\beta$ satisfy
\begin{align}\label{continuity of free additive convolution}
 \dL(\mu_A\boxplus\mu_B,\mu_\alpha\boxplus\mu_\beta)\le \dL(\mu_A,\mu_\alpha)+\dL(\mu_B,\mu_\beta)\,,
\end{align}
where $\dL$ denotes the L\'evy distance. In particular, weak convergence of $\mu_A$ to $\mu_\alpha$ and weak convergence of $\mu_B$ to $\mu_\beta$ imply weak convergence of $\mu_A\boxplus\mu_B$ to $\mu_\alpha\boxplus\mu_\beta$.

Using the Stieltjes transform, we can easily link~\eqref{continuity of free additive convolution} to the systems of equations in~\eqref{le definiting equations}, respectively in~\eqref{le H system defs}. Using integration by parts and the definition of the Stieltjes transform, a direct computation reveals that there is a numerical constant $C$ such that
\begin{align}\label{le direct consequence of continuity}
 |m_{\mu_{A}\boxplus\mu_{B}}(z)-m_{\mu_\alpha\boxplus\mu_\beta}(z)|&\le\frac{C}{\eta}\Big(1+\frac{1}{\eta}\Big) \dL(\mu_A\boxplus\mu_B,\mu_\alpha\boxplus\mu_\beta)\nonumber\\
 &\le\frac{C}{\eta}\Big(1+\frac{1}{\eta}\Big)(\dL(\mu_A,\mu_\alpha)+\dL(\mu_B,\mu_\beta))\,,\qquad \eta=\im z\,,
\end{align}
for all $z\in\C^+$, where we used~\eqref{continuity of free additive convolution} to get the second line. Note that the estimate in~\eqref{le direct consequence of continuity} deteriorates as $\eta$ approaches the real line. Our next result strengthens~\eqref{le direct consequence of continuity} as follows. We consider the measure $\mu_\alpha\boxplus\mu_\beta$ as ``reference'' measure  (in the sense that it locates the regular bulk) while $\mu_A$, $\mu_B$ are arbitrary probability measures and show that the L\'evy distances bound $|m_{\mu_{A}\boxplus\mu_{B}}(E+\ii\eta)-m_{\mu_\alpha\boxplus\mu_\beta}(E+\ii\eta)|$ uniformly in~$\eta$, for all~$E$ inside the regular bulk of	 $\mu_\alpha\boxplus\mu_\beta$. 

\begin{theorem}\label{le theorem continuity}
Let $\mu_\alpha$ and $\mu_\beta$ be compactly supported probability measures on $\R$, and assume that neither is supported at a single point and that at least one of them is supported at more than two points. Let $\mathcal{I}\subset\mathcal{B}_{\mu_\alpha\boxplus\mu_\beta}$ be a compact non-empty interval and fix some $0<\eta_{\mathrm{M}}<\infty$. Let~$\mu_A$ and $\mu_B$ be two arbitrary probability measures on $\R$. 

Then there are constants $b>0$ and $Z<\infty$, both depending on the measures $\mu_\alpha$ and $\mu_\beta$, on the interval $\mathcal{I}$ as well as on the constant $\eta_{\mathrm{M}}$,  such that whenever
\begin{align}\label{new condition}
\dL(\mu_A,\mu_\alpha)+\dL(\mu_B,\mu_\beta)\le b
\end{align}
holds, then
\begin{align}\label{le ksv statment}
 \max_{z\in\mathcal{S}_{\mathcal{I}}(0,\eta_{\mathrm{M}})}\big|m_{\mu_A\boxplus \mu_B}(z)-m_{\mu_\alpha\boxplus\mu_\beta}(z)\big|\le Z\left(\dL(\mu_A,\mu_\alpha)+\dL(\mu_B,\mu_\beta)\right)\,,
\end{align}
holds, too.
\end{theorem}
Note that $\max_{z\in\mathcal{S}_{\mathcal{I}}(0,\eta_{\mathrm{M}})} |m_{\mu_\alpha\boxplus\mu_\beta}(z)|<\infty$ by compactness of $\mathcal{I}$ and analyticity of $m_{\mu_\alpha\boxplus\mu_\beta}$ in $\mathcal{I}$. Thus the Stieltjes-Perron inversion formula directly implies that $(\mu_{A}\boxplus\mu_{B})^{\mathrm{ac}}$ has a density, $f_{\mu_A\boxplus\mu_B}$, inside $\mathcal{I}$ and that
\begin{align}\label{le uniform pointwise}
 \max_{x\in\mathcal{I}}|f_{\mu_A\boxplus\mu_B}(x)-f_{\mu_\alpha\boxplus\mu_\beta}(x)|\le Z\left(\dL(\mu_A,\mu_\alpha)+\dL(\mu_B,\mu_\beta)\right)\,,
\end{align}
provided that~\eqref{new condition} holds, where $f_{\mu_\alpha\boxplus\mu_\beta}$ is the density of $(\mu_\alpha\boxplus\mu_\beta)^{\mathrm{ac}}$.

\begin{remark}
The estimate~\eqref{le ksv statment} was recently given by Kargin~\cite{Kargin2013} under the assumption that~\eqref{le thm lower bound on omega equation} and~\eqref{le thm linear stability equation} hold for all $z\in\mathcal{S}_{\mathcal{I}}(0,\eta_\mathrm{M})$, \ie under the assumption that the conclusions of our Theorem~\ref{thm stability} hold. It is quite surprising that one can directly set $\im z=0$ in~\eqref{le ksv statment}. As first noted by Kargin, this is due to the regularizing effect of $\omega_\alpha$, $\omega_\beta$ and to the global uniqueness of solutions to~\eqref{le definiting equations} for arbitrary probability measures.
\end{remark}

\subsubsection{Application to random matrix theory}\label{Application to random matrix theory}
We now turn to an application of Theorem~\ref{thm stability} in random matrix theory. Let $A\equiv A^{(N)}$ and $B\equiv B^{(N)}$ be two sequences of $N\times N$ deterministic real diagonal matrices, whose empirical spectral distributions are denoted by~$\mu_A$ and~$\mu_B$ respectively, \ie
\begin{align}\label{le empirical measures of A and B}
 \mu_{A}\deq\frac{1}{N}\sum_{i=1}^N\delta_{a_i}\,,\qquad\qquad\mu_{B}\deq\frac{1}{N}\sum_{i=1}^N\delta_{b_i}\,,
\end{align}
where $A=\mathrm{diag}(a_i)$, $B=\mathrm{diag}(b_i)$. The matrices $A$ and $B$ depend on $N$, but we omit this fact from the notation. Let $\omega_A$ and $\omega_B$ denote the subordination functions associated with~$\mu_A$ and~$\mu_B$ by Proposition~\ref{le prop 1}.

 We assume that there are deterministic probability measures $\mu_\alpha$ and $\mu_\beta$ on $\R$, neither of them being a single point mass, such that the empirical spectral distributions $
 \mu_A, \mu_B$ converge weakly to $\mu_\alpha$, $ \mu_\beta$, as $N\to \infty$. More precisely, we assume that 
\begin{align}\label{le assumptions convergence empirical measures}
\dL(\mu_A,\mu_\alpha)+\dL(\mu_B,\mu_\beta)\to 0\,,
\end{align}
as $N\to\infty$. Let $\omega_\alpha$, $\omega_\beta$ denote the subordination functions associated with  $\mu_\alpha$ and $\mu_\beta$.

Let $U$ be an independent $N\times N$ Haar distributed unitary matrix (in short {\it Haar unitary}) and consider the random matrix
\begin{eqnarray}\label{le our H}
H\equiv H^{(N)}\deq A+UBU^*\,.
\end{eqnarray}
We introduce the {\it Green function}, $G_H$, of $H$ and its normalized trace, $m_H$, by setting
\begin{align}\label{le true green function}
 G_H(z)\deq \frac{1}{H-z}\,,\qquad\qquad m_H(z)\deq\ntr G_H(z)\,,
\end{align}
$z\in\C^+$. We refer to $z$ as the {\it spectral parameter} and we often write $z=E+\ii\eta$, $E\in\R$, $\eta> 0$. Recall the definition of $S_{\mathcal{I}}(a,b)$ in~\eqref{le domain S}. We have the following {\it local law} for $m_H$.

\begin{theorem} \label{thm041801}
Let $\mu_\alpha$ and $\mu_\beta$ be two compactly supported probability measures on $\R$, and assume that neither is only supported at one point and that at least one of them is supported at more than two points. Let $\mathcal{I}\subset\mathcal{B}_{\mu_\alpha\boxplus\mu_\beta}$ be a compact non-empty interval and fix some $0<\eta_{\mathrm{M}}<\infty$. Assume that the sequences of matrices $A$ and $B$ in~\eqref{le our H} are such that their empirical eigenvalue distributions $\mu_A$ and $\mu_B$ satisfy~\eqref{le assumptions convergence empirical measures}. Fix any small $\gamma>0$ and set $\eta_\mathrm{m}\deq N^{-2/3+\gamma}$.

 Then we have the following uniform estimate: For any (small) $\epsilon>0$ and any (large) $D$,
\begin{align}\label{le local convergence of m}
\mathbb{P}\,\bigg(\bigcup_{z\in\mathcal{S}_{\mathcal{I}}(\eta_\mathrm{m},\eta_{\mathrm{M}})}\bigg\{\big|m_H(z)-m_{\mu_A\boxplus \mu_B}(z)\big|> \frac{N^{\epsilon}}{N\eta^{3/2}}\bigg\}\bigg)\le \frac{1}{N^D}\,,
\end{align}
holds for $N\ge N_0$, with some $N_0$ sufficiently large, where we write $z=E+\ii\eta$.

\end{theorem}
Using standard techniques of random matrix theory, we can translate the estimate~\eqref{le local convergence of m} on the Green function into an estimate on the empirical spectral distribution of the matrix~$H$. Let $\lambda_1,\ldots,\lambda_N$ denote the ordered eigenvalues of $H$ and denote by
\begin{align}\label{le empirical distribution lambdas}
 \mu_H\deq \frac{1}{N}\sum_{i=1}^N\delta_{\lambda_i}
\end{align}
its empirical spectral distribution. Our result on the rate of convergence of~$\mu_H$ is as follows.	

\begin{corollary}\label{cor081801}
Let $\mathcal{I}\subset\mathcal{B}_{\mu_\alpha\boxplus\mu_\beta}$ be a compact non-empty interval. Then, for any $E_1<E_2$ in $\mathcal{I}$, we have the following estimate. For any (small) $\epsilon>0$ and any (large) $D$ we have
\begin{align}
\P\,\left(\Big|\mu_H([E_1,E_2))-\mu_{A\boxplus B}([E_1,E_2))  \Big|> \frac{N^\epsilon}{N^{2/3}}\right)\le N^{-D}\,,
\end{align}
for $N\ge N_0$, with some $N_0$ sufficiently large.
\end{corollary}
We omit the proof of Corollary~\ref{cor081801} from Theorem~\ref{thm041801}, but mention that the normalized trace $m_H$ of the Green function and the empirical spectral distribution $\mu_H$ of $H$ are linked~by
\begin{align*}
 m_H(z)=\ntr G_H(z)=\frac{1}{N}\sum_{i=1}^N\frac{1}{\lambda_i-z}=\int_\R\frac{\dd\mu_H(x)}{x-z}\,,\qquad \qquad z\in\C^+\,.
\end{align*}
Corollary~\ref{cor081801} then follows from a standard application of the Helffer-Sj\"{o}strand functional calculus; see~\eg Section~7.1 of~\cite{EKYY13} for a similar argument.

 Note that assumption~\eqref{le assumptions convergence empirical measures} does not exclude that the matrix $H$ has outliers in the large~$N$ limit. In fact, the model $H=A+UBU^*$ shows a rich phenomenology when, say,~$A$ has a finite number of large spikes; we refer to the recent works in~\cite{BBCF,BN11,C13,Kargin}. 

\begin{remark} Our results in Theorem~\ref{thm041801} and~Corollary~\ref{cor081801} are stated for $U$ Haar distributed on the unitary group $U(N)$. However, they also hold true (with the same proofs) when $U$ is Haar distributed on the orthogonal group~$O(N)$.
\end{remark}

\begin{remark}\label{remark:companion}
 In~\cite{BES15-2}, we derive, with a different approach, the estimate (with the notation of~\eqref{le local convergence of m})
 \begin{align}\label{le other local convergence of m}
\mathbb{P}\,\bigg(\bigcup_{z\in\mathcal{S}_{\mathcal{I}}(\eta_\mathrm{m},\eta_{\mathrm{M}})}\bigg\{\big|m_H(z)-m_{\mu_A\boxplus \mu_B}(z)\big|> \frac{N^{\epsilon}}{\sqrt{N\eta}}\bigg\}\bigg)\le \frac{1}{N^D}\,,
\end{align}
for $N\ge N_0$, with some $N_0$ sufficiently large, and with $\eta_{\mathrm{m}}=N^{-1+\gamma}$. In fact, we obtain estimates for individual matrix elements of the resolvent $G_{H}$ as well. Comparing with~\eqref{le local convergence of m}, we see that we can choose $\eta$ in~\eqref{le other local convergence of m} almost as small as~$N^{-1}$ at the price of losing a factor~$\sqrt{N}$.  The stability and perturbation analysis in~\cite{BES15-2} rely on the optimal results in Theorem~\ref{thm stability} and Theorem~\ref{le theorem continuity} as well as in Sections~\ref{s. stability}-\ref{section theorem continuity} of the present paper.

\end{remark}

\subsection{Organization of the paper}
In Section~\ref{s. stability}, we consider the stability of the system~\eqref{le definiting equations} when at least one of the measures $\mu_1$ and $\mu_2$ is supported at more than two points and we give the proof of Theorem~\ref{thm stability}. In Section~\ref{section perturbation}, we consider perturbations of the system~\eqref{le definiting equations} and derive results that will be used in the proof of Theorem~\ref{thm041801} and also in~\cite{BES15-2}. In Section~\ref{section theorem continuity}, we prove Theorem~\ref{le theorem continuity}. In Section~\ref{le section proof of theorem matrix} we consider the random matrix setup and prove Theorem~\ref{thm041801}. In the final Section~\ref{le final section}, we separately settle the special case when both~$\mu_1$ and~$\mu_2$ are combinations of two point masses and give the results analogous to Theorem~\ref{thm stability}, Theorem~\ref{le theorem continuity} and Theorem~\ref{thm041801} for that case.

\section{Stability of the system~\eqref{le H system} and proof of Theorem~\ref{thm stability}}\label{s. stability}
In this section, we discuss stability properties of the system~\eqref{le H system}, with $\mu_1$, $\mu_2$ two compactly supported probability measures satisfying the assumptions of Theorem~\ref{thm stability}. 

\begin{lemma}\label{le lemma stability constants}
Let $\mu_1$, $\mu_2$ be two probability measures on $\R$ neither of them being supported at a single point. Then there is, for any compact set $\mathcal{K}\subset\C^+$, a strictly positive constant $0<\sigma(\mu_1,\mathcal{K})<1$ such that the reciprocal Stieltjes transform~$F_{\mu_1}$ (see~\eqref{le F definition}) satisfies
\begin{align}\label{le first stability constant}
\im z\le (1-\sigma(\mu_1,\mathcal{K}))\,\im F_{\mu_1}(z)\,,\qquad\qquad \forall z\in\mathcal{K}\,.
\end{align}
 Similarly, there is $0<\sigma(\mu_2,\mathcal{K})<1$ such that ~\eqref{le first stability constant} holds with $\mu_2$ and $F_{\mu_2}$, respectively. 

 Assume in addition that $\mu_1$ is supported at more than two points. Then there is, for any compact set $\mathcal{K}\subset\C^+$, a strictly positive constant $0<\widetilde\sigma(\mu_1,\mathcal{K})<1$ such that
\begin{align}\label{le second stability constant}
|F_{\mu_1}'(z)-1|\le(1- \widetilde\sigma(\mu_1,\mathcal{K}))\,\frac{\im F_{\mu_1}(z)-\im z}{\im z}\,,\qquad\qquad \forall z\in\mathcal{K}\,,
\end{align}
where $F'_{\mu_1}(z)\equiv\partial_z F_{\mu_1}(z)$.
\end{lemma}

\begin{proof}[Proof of Lemma~\ref{le lemma stability constants}]
Assuming by contradiction that inequality~\eqref{le first stability constant} saturates (with vanishing constant $\sigma(\mu_1,\mathcal{K})=0$, for some $z\in \mathcal{K}\subset \C^+$), we have $\im F_{\mu_1}(z)=\im z$ for some $z$, thus $F_{\mu_1}(z)=z-a$, $a\in\R$, \ie $\mu_1=\delta_a$. This shows~\eqref{le first stability constant}.

To establish~\eqref{le second stability constant}, we first note that the analytic functions $F_{\mu_j}\,:\,\C^+\rightarrow \C^+$, $j=1,2$, have the Nevanlinna representations
\begin{align}\label{le neva for F}
 F_{\mu_j}(z)=a_{F_{\mu_j}}+z+\int_\R\frac{1+zx}{x-z}\,\dd\rho_{F_{\mu_j}}(x)\,,\qquad\qquad j=1,2\,,\qquad z\in\C^+\,,
\end{align}
where $a_{F_{\mu_j}}\in\R$ and $\rho_{{\mu_{F_j}}}$ are finite Borel measures on $\R$. Note that the coefficients of $z$ on the right-hand side are determined by~\eqref{le F behaviour at infinity}. From~\eqref{le neva for F} we see that
\begin{align}
|F_{\mu_1}'(z)-1|= \left|\int_\R\frac{1+x^2}{(x-z)^2}\,\dd\rho_{F_{\mu_1}}(x) \right|\,,\qquad\qquad z\in \C^+\,,
\end{align}
as well as
\begin{align}
\frac{\im F_{\mu_1}(z)-\im z }{\im z}=\int_\R\frac{1+x^2}{|x-z|^2}\,\dd\rho_{F_{\mu_1}}(x) \,,\qquad\qquad z\in \C^+\,.
\end{align}
Hence, assuming by contradiction that inequality~\eqref{le second stability constant} saturates (with~$\widetilde\sigma(\mu_1,\mathcal{K})=0$, for some $z\in \mathcal{K})$, we must have
\begin{align}
 \int_\R\frac{1+x^2}{|x-z|^2}\,\dd\rho_{F_{\mu_1}}(x) =\left|\int_\R\frac{1+x^2}{(x-z)^2}\,\dd\rho_{F_{\mu_1}}(x) \right|\,,
\end{align}
for some $z\in \mathcal{K}$, implying that $\rho_{F_{\mu_1}}$ is either a single point mass or $\rho_{F_{\mu_1}}=0$. In the latter case, we have $F_{\mu_1}(z)=a_{\mu_1}+z $ and we conclude that $\mu_1$ must be single point measure, but this is excluded by assumption. Thus $\rho_{F_{\mu_1}}$ is a single point mass, \ie there is a constant $d_{\mu_1}\in\R$ such that $F_{\mu_1}(z)=a_{F_{\mu_1}}+z+(1+zd_{\mu_1}^{2})/(d_{\mu_1}-z)$, $z\in \mathcal{K}$. It follows that $\mu_1$ is a convex combination of two point measures yielding a contradiction. This shows~\eqref{le second stability constant}. 
\end{proof}

\subsection{Bounds on the subordination functions}
Let $\mu_1$, $\mu_2$ be as above and let $\omega_1(z)$, $\omega_2(z)$ be the associated subordination functions. Recall that we rewrite the defining equations~\eqref{le definiting equations} for $\omega_1$ and $\omega_2$ in the compact form $\PP_{\mu_1,\mu_2}(\omega_1,\omega_2,z)=0$ introduced in~\eqref{le H system}.

We first provide upper bounds on the subordination functions~$\omega_1(z)$, $\omega_2(z)$. Our proof relies on the assumption that $\mu_1$, $\mu_2$ are compactly supported, \ie that there is a constant $L<\infty$ such that
\begin{align}\label{le compact support assumption}
 \mathrm{supp}\,\mu_1\subset [-L,L]\,,\qquad \qquad\mathrm{supp}\,\mu_2\subset[-L,L]\,.
 \end{align}
Recall from Theorem~\ref{thm stability} that we fixed a compact non-empty interval $\mathcal{I}\subset\mathcal{B}_{\mu_1\boxplus\mu_2}$. Since the density~$f_{\mu_1\boxplus\mu_2}$ is real analytic inside the regular bulk by Proposition~\ref{le real analytic prop} and since $\mathcal{I}$ is compact, there exists a constant $\kappa_0>0$ such that
 \begin{align}\label{assumption lower bound on density}
 0<\kappa_0\le\min_{x\in \mathcal{I}}f_{\mu_{1}\boxplus\mu_{2}}(x)\,.
\end{align}
Fixing a constant $0<\eta_{\mathrm{M}}<\infty$, it further follows that there is a constant $M<\infty$ such that
\begin{align}\label{assumption bounded m}
 \max_{z\in\mathcal{S}_{\mathcal{I}}(0,\eta_{\mathrm{M}})} |m_{\mu_{1}\boxplus\mu_{2}}(z)|\le M\,.
\end{align}

\begin{lemma}\label{le lem upper bound on omega}
Let $\mu_1$, $\mu_2$ be two compactly supported probability measures on~$\R$ satisfying~\eqref{le compact support assumption}, for some $L<\infty$, and assume that both are supported at more than one point. Let $\mathcal{I}\subset \mathcal{B}_{\mu_1\boxplus\mu_2}$ be a compact non-empty interval.
Then there is a constant $K<\infty$ such that
\begin{align}\label{le lem upper bound on omega equation}
 \max_{z\in\mathcal{S}_{\mathcal{I}}(0,\eta_{\mathrm{M}})} |\omega_1(z)|\le\frac{K}{2} \,,\qquad  \max_{z\in\mathcal{S}_{\mathcal{I}}(0,\eta_{\mathrm{M}})} |\omega_2(z)|\le \frac{K}{2}\,.
\end{align}
The constant $K$ depends on the constant ~$\eta_{\mathrm{M}}$ and on the interval $\mathcal{I}$ as well as  on the measures $\mu_1$, $\mu_2$ through the constants $\kappa_0$ in~\eqref{assumption lower bound on density} and the constant $L$ in~\eqref{le compact support assumption}.
\end{lemma}
\begin{proof}
 We start by noticing that there is a constant $\kappa_1>0$ such that
 \begin{align}\label{le mimimi}
 \im m_{\mu_1\boxplus\mu_2}(z)=\int_\R\frac{\eta\,\dd( \mu_1\boxplus\mu_2)(x)}{(x-E)^2+\eta^2}\ge\int_\mathcal{I}\frac{\eta\, f_{\mu_1\boxplus\mu_2}(x)\dd x}{(x-E)^2+\eta^2}\ge \kappa_1\,,
\end{align}
uniformly in $z=E+\ii\eta\in\mathcal{S}_{\mathcal{I}}(0,\eta_{\mathrm{M}}) $, where we used~\eqref{assumption lower bound on density}. Thus by subordination we have
 \begin{align}\label{le 44}
  \min_{z\in\mathcal{S}_{\mathcal{I}}(0,\eta_{\mathrm{M}})}|m_{\mu_1\boxplus\mu_2}(z)|=\min_{z\in\mathcal{S}_{\mathcal{I}}(0,\eta_{\mathrm{M}})}\left|\int_\R\frac{\dd\mu_1(a)}{a-\omega_2(z)} \right|\ge \kappa_1\,,
 \end{align}
 since $m_{\mu_1\boxplus\mu_2}(z)=-1/F_{\mu_1\boxplus\mu_2}(z)$ by~\eqref{le kkv}.
 
On the other hand, $\mu_1$ is supported on the interval $[-L,L]$; see~\eqref{le compact support assumption}. Hence, using~\eqref{le 44}, $|\omega_2(z)|$ must be bounded from above on $\mathcal{S}_{\mathcal{I}}(0,\eta_{\mathrm{M}})$.  Interchanging the r\^{o}les of the indices $1$ and $2$, we also get that $|\omega_1(z)|$ is bounded from above on $\mathcal{S}_{\mathcal{I}}(0,\eta_{\mathrm{M}})$. \qedhere
\end{proof}

Having established upper bounds on the subordination functions, we show that their imaginary parts are uniformly bounded from below on the domain $\mathcal{S}_{\mathcal{I}}(0,\eta_{\mathrm{M}})$. The proof relies on inequality~\eqref{le first stability constant}.

\begin{lemma}\label{le lemma stability}
Let $\mu_1$, $\mu_2$ be two probability measures on~$\R$ satisfying~\eqref{le compact support assumption}, for some $L<\infty$, and assume that neither of them is only supported at a single point. Let $\mathcal{I}\subset \mathcal{B}_{\mu_1\boxplus\mu_2}$ be a compact non-empty interval. Then there is a strictly positive constant $k>0$ such that
\begin{align}\label{le stability}
  \min_{z\in \mathcal{S}_{\mathcal{I}}(0,\eta_{\mathrm{M}})}\im \omega_1(z)\ge 2k\,,\qquad\qquad \min_{z\in \mathcal{S}_{\mathcal{I}}(0,\eta_{\mathrm{M}})}\im \omega_2(z)\ge 2k\,.
\end{align}
\end{lemma}
\begin{remark}\label{le remark for k}
 The constant $k$ in~\eqref{le stability} depends on the interval $\mathcal{I}$ through the constants $\kappa_0$ in~\eqref{assumption lower bound on density} and $M$ in~\eqref{assumption bounded m}. It further depends on~$\eta_{\mathrm{M}}$, as well as on $\sigma(\mu_1,\mathcal{K}_2)$ and $\sigma(\mu_2,\mathcal{K}_1)$ in~\eqref{le first stability constant}, with 
 \begin{align}\label{le kevin set}
 \mathcal{K}_i=\{u\in\C^+\,:\, u=\omega_{i}(z)\,, z\in\mathcal{S}_{\mathcal{I}}(0,\eta_{\mathrm{M}})\}\,,\qquad i=1,2\,.
 \end{align}
\end{remark}

\begin{proof}[Proof of Lemma~\ref{le lemma stability}]
First note that there is $\kappa_1>0$ such that $\im m_{\mu_1\boxplus\mu_2}(z)\ge \kappa_1$ for all $z\in \mathcal{S}_{\mathcal{I}}(0,\eta_{\mathrm{M}})$; \cf~\eqref{le mimimi}. Moreover, there is $M<\infty$ such that $|m_{\mu_1\boxplus\mu_2}(z)|\le M$ for all ${z\in\mathcal{S}_{\mathcal{I}}(0,\eta_{\mathrm{M}})}$; \cf~\eqref{assumption bounded m}. Recall from~\eqref{le definiting equations} and~\eqref{le kkv} that
\begin{align}
\omega_1(z)+\omega_2(z)=z-\frac{1}{m_{\mu_1\boxplus\mu_2}(z)}\,, \qquad \qquad z\in\C^+\,.
\end{align}
Hence, considering the imaginary part, we notice from~\eqref{assumption bounded m} that there is $\kappa_2>0$ such that 
\begin{align}
\min_{z\in \mathcal{S}_{\mathcal{I}}(0,\eta_{\mathrm{M}})}(\im \omega_1(z)+\im \omega_2(z))\ge \kappa_2\,.
 \end{align}

It remains to show that $\im \omega_1$ and $\im\omega_2$ are separately bounded from below. To do so we invoke~\eqref{le first stability constant} and assume by contradiction that $\im \omega_1(z)\le \epsilon$, for some small $0\le\epsilon<\kappa_2/2$. We must thus have $\im \omega_2(z)\ge \kappa_2/2$. Since $\mu_1$ is assumed not to be a single point mass,  Lemma~\ref{le lemma stability constants} assures that 
\begin{align}\label{le esel2}
\im F_{\mu_1}(\omega_2(z))\ge \frac{\im \omega_2(z)}{1-\sigma(\mu_1,\mathcal{K}_2)} \,,\qquad\qquad z\in\mathcal{S}_{\mathcal{I}}(0,\eta_{\mathrm{M}})\,,	
\end{align}
with $0<\sigma(\mu_1,\mathcal{K}_2)<1$, where $\mathcal{K}_2$ denotes the image of $\mathcal{S}_{\mathcal{I}}(0,\eta_{\mathrm{M}})$ under the map $\omega_2$ (which is necessarily compact by Lemma~\ref{le lem upper bound on omega}). On the other hand,~\eqref{le H system} implies
\begin{align}\label{le esel1} 
\im F_{\mu_1}(\omega_2(z))=\im \omega_2(z)+\im \omega_1(z)-\im z\,,\qquad \qquad z\in\C^+\,.
\end{align}
Since $\im \omega_1(z)\ge \im z$, by Proposition~\ref{le prop 1}, we get, by comparing~\eqref{le esel1} and~\eqref{le esel2}, a contradiction with the assumption that $\im \omega_1(z)\le \epsilon$, for sufficiently small $\epsilon$. Repeating the argument with the r\^{o}les of the indices $1$ and $2$ interchanged, we get~\eqref{le stability}.\qedhere
\end{proof}

\subsection{Linear stability of~\eqref{le H system}}
Having established lower and upper bounds on the subordination functions $\omega_1$, $\omega_2$, we now turn to the stability of the system $\PP_{\mu_1,\mu_2}(\omega_1,\omega_2,z)=0$. Remember that we call the system linearly $S$-stable at $(\omega_1,\omega_2)$ if $\Gamma_{\mu_1,\mu_2}(\omega_1,\omega_2)\le S$, where~$\Gamma_{\mu_1,\mu_2}$ is defined in~\eqref{le what stable means}.

\begin{lemma}\label{lemma linear stability}
Let $\mu_1$, $\mu_2$ be two probability measures on~$\R$ satisfying~\eqref{le compact support assumption} for some $L<\infty$. Assume that neither of them is a single point mass and that at least one of them is supported at more than two points. Let $\mathcal{I}\subset \mathcal{B}_{\mu_1\boxplus\mu_2}$ be a compact non-empty interval.

Then, there is a finite constant $S$ such that
\begin{align}\label{le linear stability equation}
\max_{z\in\mathcal{S}_{\mathcal{I}}(0,\eta_{\mathrm{M}})}\Gamma_{\mu_1,\mu_2}(\omega_1(z),\omega_2(z))\le S\,,
\end{align}
and 
\begin{align}\label{le control of omega derivata}
\max_{z\in\mathcal{S}_{\mathcal{I}}(0,\eta_{\mathrm{M}})}|\omega_1'(z)|\le 2S\,,\qquad\qquad \max_{z\in\mathcal{S}_{\mathcal{I}}(0,\eta_{\mathrm{M}})}|\omega_2'(z)|\le 2S\,,
\end{align}
where $\omega_1(z)$, $\omega_2(z)$ are the solutions to $\PP_{\mu_1,\mu_2}(\omega_1,\omega_2,z)=0$.
\end{lemma}
\begin{remark}
Lemma~\ref{lemma linear stability} is the first instance where we use that at least one of $\mu_1$ and $\mu_2$ is supported at more than two points.  For definiteness, we assume that $\mu_1$ is supported at more than two points. The constant $S$ in~\eqref{le linear stability equation} depends on the interval $\mathcal{I}$ through the constant $\kappa_0$ in~\eqref{assumption lower bound on density}, on the constants~$\eta_{\mathrm{M}}$,~$L$ in~\eqref{le compact support assumption},~$\sigma(\mu_1,\mathcal{K}_2)$ and $\sigma(\mu_2,\mathcal{K}_2)$, as well as on $\widetilde\sigma(\mu_1,\mathcal{K}_2)$ of~\eqref{le second stability constant} with $\mathcal{K}_2$ defined in~\eqref{le kevin set}.
\end{remark}

\begin{proof}[Proof of Lemma~\ref{lemma linear stability}]
 Using~\eqref{le what stable means} and  Cramer's rule, $\Gamma\equiv\Gamma_{\mu_1,\mu_2}(\omega_1,\omega_2,z) $ equals\small
\begin{align}\label{le gamma for the second}
 \Gamma=\frac{1}{\left|1-(F_{\mu_1}'(\omega_2)-1) (F_{\mu_2}'(\omega_1)-1) \right|}\left\|\left( \begin{array}{cc}
 -1& -F_{\mu_1}'(\omega_2(z))+1  \\
 -F_{\mu_2}'(\omega_1(z))+1 & -1\\
   \end{array} \right) \right\|\,.
\end{align}\normalsize
As above, we assume for definiteness that $\mu_1$ is supported at more than two points. We first focus on $F'_{\mu_1}(\omega_2)$. Recalling the definition of $\mathcal{K}_2$ from Remark~\ref{le remark for k} and invoking~\eqref{le second stability constant}, we obtain
\begin{align}
 |F'_{\mu_1}(\omega_2(z))-1| \le \big( 1-\widetilde\sigma(\mu_1,\mathcal{K}_2)\big) \frac{\im F_{\mu_1}(\omega_2(z))-\im\omega_2(z)}{\im \omega_2(z)} \,,
\end{align}
for all $z\in\mathcal{S}_{\mathcal{I}}(0,\eta_{\mathrm{M}})$, where $0<\widetilde\sigma(\mu_1,\mathcal{K}_2)<1$. Abbreviating $\widetilde\sigma\equiv\widetilde\sigma(\mu_1,\mathcal{K}_2)$ and using $\PP_{\mu_1,\mu_2}(\omega_1(z),\omega_2(z),z)=0$, we thus have
\begin{align}\label{le 413}
  |F'_{\mu_1}(\omega_2(z))-1| \le \left(1-\widetilde\sigma\right)\frac{\im \omega_1(z)}{\im \omega_2(z)}\,,\qquad\qquad z\in \mathcal{S}_{\mathcal{I}}(0,\eta_{\mathrm{M}})\,.
\end{align}
Reasoning in the similar way (\cf~\eqref{le fbound}), we also obtain
\begin{align}\label{le 414}
 |F'_{\mu_2}(\omega_1(z))-1|\le \frac{\im \omega_2(z)}{\im \omega_1(z)}\,,\qquad \qquad z\in\C^+\,,
\end{align}
where the inequality may saturate here since we do not exclude $\mu_2$ being supported at two points only. Multiplying~\eqref{le 413} and~\eqref{le 414}, we get
\begin{align}\label{le bound on determinant}
 \max_{z\in\mathcal{S}_{\mathcal{I}}(0,\eta_{\mathrm{M}})}|F'_{\mu_1}(\omega_2(z))-1|\, |F'_{\mu_2}(\omega_1(z))-1|\le  1-\widetilde\sigma\,.
\end{align}
Using Lemma~\ref{le lem upper bound on omega} and Lemma~\ref{le lemma stability}, we also have from~\eqref{le 413} and~\eqref{le 414} that
\begin{align}\label{le bound on the entries of gamma}
\max_{z\in\mathcal{S}_{\mathcal{I}}(0,\eta_{\mathrm{M}})}|F'_{\mu_i}(\omega_j(z))-1|\le \frac{K}{4k}\,,\qquad\qquad \{i,j\}=\{1,2\}\,.
\end{align}
Hence, bounding the operator norm by the Hilbert-Schmidt norm in~\eqref{le gamma for the second}, we obtain by~\eqref{le bound on the entries of gamma} and~\eqref{le bound on determinant} that
\begin{align}
 \max_{z\in\mathcal{S}_{\mathcal{I}}(0,\eta_{\mathrm{M}})}\Gamma_{\mu_1,\mu_2}(\omega_1(z),\omega_2(z))\le \frac{\sqrt{2}}{\widetilde\sigma}\left(1+\Big(\frac{K}{4k} \Big)^2\right)^{1/2}=:S\,,
\end{align}
with finite constant $S$. This proves~\eqref{le linear stability equation}.	

The estimates in~\eqref{le control of omega derivata} follow by differentiating the equation $\PP_{\mu_1,\mu_2}(\omega_1(z),\omega_2(z),z)=0$ with respect to $z$. We get
\begin{align}\label{le baby system}\left( \begin{array}{cc}
 -1& F_{\mu_1}'(\omega_2(z))-1  \\
 F_{\mu_2}'(\omega_1(z))-1 & -1\\
   \end{array} \right)\,\left(\begin{array}{c}\omega_1'(z)\\ \omega_2'(z) \end{array}\right)=\left(\begin{array}{c}1\\ 1 \end{array}\right)\,.
\end{align}
From~\eqref{le linear stability equation} we know that $\PP$ is uniformly $S$-stable and we get~\eqref{le control of omega derivata} by inverting~\eqref{le baby system}.\qedhere
\end{proof}

\begin{remark}
 The crucial estimate in the proof above is~\eqref{le bound on determinant}. An alternative proof of~\eqref{le bound on determinant} under the assumptions that both $\mu_1$ and $\mu_2$ are supported at more than two points was pointed out by  
 an anonymous referee. From~\eqref{le definiting equations} we observe that the subordination function~$\omega_1(z)$ appears, for fixed $z\in\C^+$, as the fixed point of the map $\mathcal{F}_z\,:\C^+\rightarrow \C^+$,
 \begin{align}\label{the calF map}
  u\mapsto \mathcal{F}_z(u)\deq F_{\mu_1}(F_{\mu_2}(u)-u+z)-F_{\mu_2}(u)-u+z\,.
 \end{align}
  Indeed, assuming that $\omega_1\in\C^+$ and that $\mu_1$, $\mu_2$ are supported at least at three points (so that $F_{\mu_1}(z)-z$ and $F_{\mu_2}(z)-z$ are not M\"obius transformations), the fixed point $\omega_1(z)$ is attracting as was shown in~\cite{BB}. Thus, for any fixed $k>0$, the Schwarz--Pick Theorem and~\eqref{le definiting equations} imply that for any compact subset $\widehat{\mathcal{K}}$ of $\{z\in\C^+\cup\R\,:\, \im \omega_1(z),\im \omega_2(z)\ge 2k\}$ there is a constant $\widehat{\sigma}(\widehat{\mathcal{K}})<1$ such that $|F'_1(\omega_2(z)-1||F'_2(\omega_1(z)-1|\le \widehat{\sigma}(\widehat{\mathcal{K}})<1$, for any $z\in\widehat{ \mathcal{K}}$. Thus, under the assumption that $\mu_1$ and $\mu_2$ are both supported at least at three points,~\eqref{le bound on determinant} follows from Lemma~\ref{le lem upper bound on omega} and Lemma~\ref{le lemma stability}.
 
\end{remark}

Collecting the results of this section, we obtain the proof of Theorem~\ref{thm stability}.
\begin{proof}[Proof of Theorem~\ref{thm stability}]
Lemma~\ref{le lemma stability} proves~\eqref{le thm lower bound on omega equation}. Lemma~\ref{lemma linear stability} proves~\eqref{le thm linear stability equation}. 
\end{proof}

\section{Perturbations of the system~\eqref{le H system}}\label{section perturbation}
In this section, we study perturbations of the system $\PP_{\mu_1,\mu_2}(\omega_1,\omega_2,z)=0$, where $\mu_1$, $\mu_2$ denote general compactly supported probability measures on $\R$. The main results of this section, Proposition~\ref{le proposition perturbation of system} below, is used repeatedly in the continuity argument to prove Theorem~\ref{thm041801}. Yet, as noted in Corollary~\ref{corollary perturbation}, it is of interest itself and it is also used in~\cite{BES15-2}.

\begin{proposition}\label{le proposition perturbation of system}
Fix $z_0\in\C^+$. Assume that the functions $\widetilde\omega_1$, $\widetilde\omega_2$, $\widetilde{r}_1$, $\widetilde{r}_2\,:\,\C^+\rightarrow \C$ satisfy $\im\widetilde\omega_1(z_0)>0$, $\im\widetilde\omega_2(z_0)>0$ and 
 \begin{align}\label{la perturbed system}
 \PP_{\mu_1,\mu_2}(\widetilde\omega_1(z_0),\widetilde\omega_2(z_0),z_0)=\widetilde r(z_0)\,,
 \end{align}
where $\widetilde{r}(z)\deq(\widetilde r_1(z),\widetilde r_2(z))^\top$. Assume moreover that there is $\delta\in[0,1]$ such that
 \begin{align}\label{le apriori closeness}
 |\widetilde\omega_1(z_0)-\omega_1(z_0)|\le \delta\,,\qquad |\widetilde\omega_2(z_0)-\omega_2(z_0)|\le \delta\,,  
 \end{align}
 where $\omega_1(z)$, $\omega_2(z)$ solve the unperturbed system $\PP_{\mu_1,\mu_2}(\omega_1,\omega_2,z)=0$. Assume that there is a constant $S$ such that $\PP$ is linearly $S$-stable at $(\omega_1(z_0),\omega_2(z_0))$, and assume in addition that there are strictly positive constants $K$ and $k$ with $k>\delta$ and with $k^2>\delta KS$ such that
\begin{align}\label{le stability sum ims}
0<2k\le \im \omega_1(z_0)\le K \,,\qquad\quad 0<2k\le \im \omega_2(z_0)\le K \,.
 \end{align}
 
Then we have the bounds
\begin{align}\label{le conclusion of lemma}
 |\widetilde\omega_1(z_0)-\omega_1(z_0)|\le  2S \|\widetilde{r}(z_0)\|\,,\qquad|\widetilde\omega_2(z_0)-\omega_2(z_0)|\le  2S \|\widetilde{r}(z_0)\|\,.
\end{align} 
\end{proposition}

\begin{proof}
 Combining~\eqref{le stability sum ims} and~\eqref{le apriori closeness} with~$\delta<k$, we get
 \begin{align}\label{le stability 2}
  \im \widetilde\omega_1(z_0)\ge k\,,\qquad\qquad \im \widetilde\omega_2(z_0)\ge k\,.
  \end{align}
Next, we bound higher derivatives of $F_i\equiv F_{\mu_i}$, $i=1,2$. We first note that by the Nevanlinna representation~\eqref{le neva for F} we have
\begin{align}\label{le rhb}
 \frac{\im F_i(\omega)}{\im \omega}= 1+\int_\R\frac{1+x^2}{|x-\omega |^2}\,\dd\rho_{F_i}(x)\,,\qquad\qquad \omega\in\C^+\,,\qquad i=1,2\,.
\end{align}
On the other hand, we also have from~\eqref{le neva for F} that
\begin{align}\label{le derivative F finnish}
 |F_i'(\omega)-1|\le\int_\R\frac{1+x^2}{|x-\omega|^2}\,\dd\rho_{F_i}(x)\,,\qquad\qquad \omega\in\C^+\,,\qquad i=1,2\,,
\end{align}
and analogously for higher derivatives, $n\ge 1$,
\begin{align}\label{le higher derivatives F finnish}
 |F_i^{(n)}(\omega)|\le\int_\R\frac{1+x^2}{|x-\omega|^{n+1}}\,\dd\rho_{F_i}(x)\le\frac{1}{(\im \omega)^{n-1}}\int_\R\frac{1+x^2}{|x-\omega|^2}\,\dd\rho_{F_i}(x)\,,
\end{align}
$\omega\in\C^+$, $i=1,2$. Thus, combining~\eqref{le derivative F finnish},~\eqref{le rhb} and~\eqref{le H system} we get
\begin{align}\label{le fbound}
 |F_i'(\omega_j(z))-1|\le\frac{\im F_i(\omega_j(z))-\im \omega_j(z)}{\im \omega_j(z)}=\frac{\im \omega_i(z)-\im z}{\im \omega_j(z)}\,,\qquad \{i,j\}=\{1,2\}\,,
\end{align}
$z\in\C^+$, and similarly, starting from~\eqref{le higher derivatives F finnish}, 
\begin{align}\label{le bound on F derivata}
 |F_i^{(n)}(\omega_j(z))|\le \frac{\im \omega_i(z)-\im z}{(\im \omega_j(z))^{n}}\,,\qquad\qquad z\in\C^+\,,\qquad \{i,j\}=\{1,2\}\,.
\end{align}

Let $\Omega_{i}(z)\deq\widetilde\omega_i(z)-\omega_i(z)$, $i=1,2$, and $\Omega\deq(\Omega_1,\Omega_2)^\top$. Fixing $z=z_0$ and Taylor expanding $F_1(\widetilde\omega_2(z_0))$ around $\omega_2(z_0)$ we get
\begin{align}\label{le full taylor polynomial}
 F'_1(\omega_2(z_0))\Omega_2(z_0)-\Omega_1(z_0)-\Omega_2(z_0)=\widetilde{r}_1(z)-\sum_{n\ge 2}\frac{1}{n!}F^{(n)}_1(\omega_2(z_0))\Omega_2(z_0)^n\,.
\end{align}
Recalling that $\|\Omega(z_0)\|/2k\le \delta/k<1$ and using~\eqref{le bound on F derivata} together with~\eqref{le stability sum ims}, we obtain from~\eqref{le full taylor polynomial} the estimate
\begin{align}\label{le full expansion of GammaA}
 \left| F'_1(\omega_2(z_0))\Omega_2(z_0)-\Omega_1(z_0)-\Omega_2(z_0)\right|\le \|\widetilde r(z_0)\|+ \frac{K}{4 k^{2}}\|\Omega(z_0)\|^2\,,
\end{align}
and the analogous expansion with the r\^{o}les of the indices $1$ and $2$ interchanged. We therefore obtain from~\eqref{le what stable means} and from solving the linearized equation that
\begin{align}
 \|\Omega(z_0)\|\le S \|\widetilde r(z_0)\|+\frac{KS}{4 k^{2}}\|\Omega(z_0)\|^2\,.
\end{align}
Thus, we have the dichotomy that either $ \|\Omega(z_0)\|\le 2 S\,\|\widetilde r\|$ or $2(KS)^{-1}k^{2}\le   \|\Omega(z_0)\|$. Since $k^2>\delta KS$ by assumption, the second alternative contradicts $\|\Omega(z_0)\|\le 2\delta$. This proves the estimates in~\eqref{le conclusion of lemma}.\qedhere
\end{proof}

In Proposition~\ref{le proposition perturbation of system} we assumed the apriori bound $|\widetilde\omega_i-\omega_i|\le \delta$; see~\eqref{le apriori closeness}. The next lemma shows that we may drop this assumption, for spectral parameters $z$ with sufficiently large imaginary part, at the price of assuming effective lower bounds on $\im \widetilde\omega_i$. This statement will be used as an initial input to start the continuity argument
in Section~\ref{le section proof of theorem matrix}.

\begin{lemma}\label{le lemma large eta}
Assume there is a (large) $\widetilde\eta_0>0$ such that for any $z\in\C^+$ with $\im z\ge \widetilde\eta_0$ the analytic functions $\widetilde\omega_1$, $\widetilde\omega_2$, $\widetilde{r}_1$, $\widetilde{r}_2\,:\,\C^+\rightarrow \C$ satisfy
\begin{align}\label{le lower bound on tilde omega}
\im \widetilde\omega_1(z)-\im z\ge 2\|\widetilde r(z)\|\,,\qquad\qquad \im \widetilde\omega_2(z)- \im z\ge 2\|\widetilde r(z)\|\,.
 \end{align}
and
\begin{align}\label{le other PP system}
  \PP_{\mu_1,\mu_2}(\widetilde\omega_1(z),\widetilde\omega_2(z),z)=\widetilde{r}(z)\,,
 \end{align}
where $\widetilde r(z)\deq(\widetilde{r}_1(z),\widetilde{r}_2(z))^\top$.
 
 Then there is a constant $\eta_0>0$, with $\eta_0 \ge \widetilde\eta_0$,  such that
 \begin{align}\label{le conclusion of large eta lemma}
 |\widetilde\omega_1(z)-\omega_1(z)|&\le 2\|\widetilde{r}(z)\|\,,\qquad\qquad|\widetilde\omega_2(z)-\omega_2(z)|\le 2\|\widetilde{r}(z)\|\,,
\end{align}
on the domain $\{z\in\C^+\,:\,\im z\ge \eta_0\}$, where $\omega_1$ and $\omega_2$ are the subordination functions associated with~$\mu_1$ and~$\mu_2$. The constant $\eta_0$ depends on the measures~$\mu_1$ and~$\mu_2$, and on the function~$\widetilde r$ through the constant $\widetilde\eta_0>0$.

\end{lemma}
\begin{proof}Recall the Nevanlinna representation~\eqref{le neva for F} for $F_{\mu_1}$ and $F_{\mu_2}$. Since $\mu_1$ and $\mu_2$ are compactly supported, we have, as $\im \omega\nearrow \infty$,
 \begin{align}\label{la compact support implication for F}
  F_{\mu_1}(\omega)-\omega=a_{1}+O(|\omega|^{-1})\,,\qquad F_{\mu_2}(\omega)-\omega=a_{2}+O(|\omega|^{-1})\,,
 \end{align}
 with $a_1\equiv a_{F_{\mu_1}}$ and $a_2\equiv a_{F_{\mu_2}}$.
There are thus $\widetilde s_1,\widetilde s_2\,:\,\C^+\to \C$ such that
\begin{align}\label{le ach}
\PP_{\mu_1,\mu_2}(\widetilde\omega_1(z),\widetilde\omega_2(z),z)=\left(\begin{array}{cc}a_{1}+\widetilde s_1(z)-\widetilde\omega_1(z)+z\\a_2+\widetilde s_2(z)-\widetilde\omega_2(z)+z \end{array} \right)=\left(\begin{array}{cc}\widetilde r_1(z)\\ \widetilde r_2(z) \end{array}\right)\,,
\end{align}
with 
\begin{align}
\widetilde s_1(z)=O(|\widetilde\omega_2(z)|^{-1})\,,\qquad\qquad \widetilde s_2(z)=O(|\widetilde\omega_1(z)|^{-1})\,,
\end{align}
as $\im z\nearrow \infty$. It follows immediately that $\widetilde\omega_1(z)=O(\im z)$ and $\widetilde\omega_2(z)=O(\im z)$, as $\im z\nearrow \infty$. Thus, recalling the definition of $\Gamma_{\mu_1,\mu_2}$ in~\eqref{le what stable means}, we get
\begin{align}
 \Gamma_{\mu_1,\mu_2}(\widetilde\omega_1(z),\widetilde\omega_2(z))=1+ O(\eta^{-2})\,,
\end{align}
as $\eta=\im z\nearrow \infty$. In particular, we obtain
\begin{align}
 \|((\mathrm{D}\PP)^{-1}\PP)(\widetilde\omega_1(z),\widetilde\omega_2(z),z)\|&\le \| \Gamma_{\mu_1,\mu_2}(\widetilde\omega_1(z),\widetilde\omega_2(z)) \| \|\PP(\widetilde\omega_1(z),\widetilde\omega_2(z),z)\|\nonumber\\
 &\le 2 \|\widetilde r(z)\|\,,
\end{align}
for $\im z$ sufficiently large. 
From~\eqref{le higher derivatives F finnish} and~\eqref{la compact support implication for F}, we also get
\begin{align}
 |F^{(2)}_{\mu_i}(\omega)|\le \frac{\im F_{\mu_i}(\omega)-\im \omega}{(\im \omega)^2}=O((\im \omega)^{-3})\,,\qquad\qquad \omega\in\C^+\,,\qquad i=1,2\,,
\end{align}
as $\im \omega\nearrow\infty$. Thus the matrix of second derivatives of $\PP$ given by
\begin{align*}
 \mathrm{D}^2\PP(\omega_1,\omega_2)\deq\left(\frac{\partial^2 \PP}{\partial\omega_1^2}(\omega_1,\omega_2,z) \,,\, \frac{\partial^2\PP}{\partial \omega_2^2}(\omega_1,\omega_2,z)\right)=\left(\begin{array}{cc}0& F_{\mu_1}^{(2)}(\omega_2) \\  F_{\mu_2}^{(2)}(\omega_1) & 0\end{array} \right)\,,
\end{align*}
satisfies $\|\mathrm{D}^2\PP(\widetilde{\omega}_1(z),\widetilde{\omega}_2(z))\|=O{(\im z)^{-3}}$, as $\im z\nearrow\infty$. Hence, choosing $ \eta_0>0$ sufficiently large, we achieve that
 \begin{align*}
 s_0\deq 2\|\widetilde r(z)\|\,\|\mathrm{D}^2\PP (\widetilde\omega_1(z),\widetilde\omega_2(z))\|<\frac{1}{2}\,,
 \end{align*}
 on the domain $\{z\in \C^+\,:\,\im z\ge \eta_0\}$. Thus, by the Newton-Kantorovich theorem (see, \eg Theorem~1 in~\cite{FS}), there are for every such $z$ unique $\widehat\omega_1(z),\widehat \omega_2(z)$ such that $\PP_{\mu_1,\mu_2}(\widehat\omega_1(z),\widehat\omega_2(z),z)=0$, with 
\begin{align}\label{le output of NK}
|\widetilde\omega_i(z)-\widehat\omega_i(z)|\le \frac{1-\sqrt{1-2s_0}}{s_0}\,\|\widetilde r(z)\|\le  2\|\widetilde r(z)\|\,,\qquad \qquad i=1,2\,.
\end{align}
Finally, we note that $\im \widehat\omega_1(z)=\im\widehat \omega_1(z)-\im \widetilde \omega_1(z)+\im\widetilde \omega_1\ge \im z$, by~\eqref{le lower bound on tilde omega}, for $\im z\ge \eta_0$. Similarly, $\im \widehat\omega_2(z)\ge \im z$, for $\im z\ge \eta_0$. It further follows that $\Gamma_{\mu_1,\mu_2}(\widehat\omega_1(z),\widehat\omega_2(z))\not=0$, for all $z\in\C^+$ with $\im z\ge \eta_0$, thus $\widehat\omega_1(z)$ and $\widehat\omega_2(z)$ are analytic on $\{z\in \C\,:\im z>\eta_0\}$ since $F_{\mu_1}$ and $F_{\mu_2}$ are. Finally, using~\eqref{la compact support implication for F} with $\omega=\widehat\omega_1$, $\omega=\widehat\omega_2$ respectively, we see that 
\begin{align*}
\lim_{\eta\nearrow\infty}\frac{\im \widehat \omega_1(\ii\eta)}{\ii\eta}=\lim_{\eta\nearrow\infty}\frac{\im \widehat \omega_1(\ii\eta)}{\ii\eta}=1\,.
\end{align*}
Thus, by the uniqueness claim in Proposition~\ref{le prop 1}, $\widehat\omega_1(z),\widehat\omega_2(z)$ agree with $\omega_1(z),\omega_2(z)$ on the domain $\{z\in\C^+\,:\, \im z\ge \eta_0\}$. This proves~\eqref{le conclusion of large eta lemma} from~\eqref{le output of NK}. \qedhere
\end{proof}

\section{Proof of Theorem~\ref{le theorem continuity}}\label{section theorem continuity}
In the setup of Theorem~\ref{le theorem continuity} we have two pairs of probability measures on $\R$, $\mu_\alpha$, $\mu_\beta$ and $\mu_A$, $\mu_B$, where we consider $\mu_\alpha$, $\mu_\beta$ as ``reference'' measures (in the sense that they satisfy the assumptions of Theorem~\ref{le theorem continuity}), while $\mu_A$, $\mu_B$ are arbitrary. Under the assumptions of Theorem~\ref{le theorem continuity} we can apply Theorem~\ref{thm stability} with the choices $\mu_\alpha=\mu_1$ and $\mu_\beta=\mu_2$.

Recall from~\eqref{le domain S} the definition of the domain $\mathcal{S}_{\mathcal{I}}(a,b)$, $a\le b$.
\begin{lemma} \label{cor.080601}
 Let $\mu_A$, $\mu_B$ and $\mu_\alpha$, $\mu_\beta$ be the probability measures from~\eqref{le empirical measures of A and B} and~\eqref{le assumptions convergence empirical measures} satisfying the assumptions of Theorem~\ref{le theorem continuity}. Let $\omega_A,\omega_B$ and $\omega_\alpha,\omega_\beta$ denote the associated subordination functions by Proposition~\ref{le prop 1}. Let $\mathcal{I}\subset\mathcal{B}_{\mu_\alpha\boxplus\mu_\beta}$ be a compact non-empty interval.  Fix $0<\eta_{\mathrm{M}}<\infty$. 
 
 Then there are a (small) constant $b_0>0$ and a (large) constant $K_1<\infty$, both depending on the measures $\mu_\alpha$ and $\mu_\beta$, on the interval $\mathcal{I}$ and on the constant $\eta_{\mathrm{M}}$, such that whenever
 \begin{align}
  \dL(\mu_A,\mu_\alpha)+\dL(\mu_B,\mu_\beta)\le b_0\,,
 \end{align}
holds, then
 \begin{align}\label{le estimate for the deterministic guys}
  |\omega_{A}(z)-\omega_\alpha(z)|&\le K_1\frac{\dL(\mu_A,\mu_\alpha)}{(\Im\omega_\beta(z))^2}+K_1\frac{\dL(\mu_B,\mu_\beta)}{(\im \omega_\alpha(z))^2}\,,\nonumber\\ |\omega_B(z)-\omega_\beta(z)|&\le K_1\frac{\dL(\mu_A,\mu_\alpha)}{(\Im\omega_\beta(z))^2}+K_1\frac{\dL(\mu_B,\mu_\beta)}{(\im \omega_\alpha(z))^2}\,,
 \end{align}
hold uniformly on $\mathcal{S}_{\mathcal{I}}(0,\eta_{\mathrm{M}})$. In particular, choosing $b\le b_0$ sufficiently small and assuming that $\dL(\mu_A,\mu_\alpha)+\dL(\mu_B,\mu_\beta)\le b $, we have, 
 \begin{align}\label{le upper bound on omega AB}
  \max_{z\in\mathcal{S}_{\mathcal{I}}(0,\eta_{\mathrm{M}})}|\omega_A(z)|\le K\,,\qquad  \max_{z\in\mathcal{S}_{\mathcal{I}}(0,\eta_{\mathrm{M}})}|\omega_B(z)|\le K\,,
 \end{align}
 \begin{align}\label{le lower bound on omega AB}
 \min_{z\in\mathcal{S}_{\mathcal{I}}(0,\eta_{\mathrm{M}})}\im\omega_A(z)\ge k\,,\qquad  \min_{z\in\mathcal{S}_{\mathcal{I}}(0,\eta_{\mathrm{M}})}\im\omega_B(z)\ge k\,,
 \end{align}
 where $K$ and $k$ are the constant from Lemma~\ref{le lem upper bound on omega} and Lemma~\ref{le lemma stability}, respectively.

\end{lemma}
\begin{remark}Armed with the conclusions of Theorem~\ref{thm stability}, our proof follows closely the arguments of~\cite{Kargin2013}. We further remark that the main argument in the proof of Lemma~\ref{cor.080601} is different from the ones given in Section~\ref{section perturbation}: it crucially relies on the global uniqueness of solutions on the upper half plane for both systems, $\PP_{\mu_\alpha,\mu_\beta}(\omega_\alpha,\omega_\beta,z)=0$ and $\PP_{\mu_A,\mu_B}(\omega_A,\omega_B,z)=0$, asserted by Proposition~\ref{le prop 1}.
\end{remark}

\begin{proof}[Proof of Lemma~\ref{cor.080601}] We first write the system $\PP_{\mu_\alpha,\mu_\beta}(\omega_\alpha(z),\omega_\beta(z),z)=0$ as 
\begin{align*}
\PP_{\mu_A,\mu_B}(\omega_\alpha(z),\omega_\beta(z),z)=r(z)\,, \qquad \qquad z\in\C^+\,,
\end{align*}
with
\begin{align}
 r(z)\equiv\left(\begin{array}{c}r_A(z)\\ r_B(z)\end{array} \right)\deq\left(\begin{array}{c}F_{\mu_A}(\omega_\beta(z))-F_{\mu_\alpha}(\omega_\beta(z)) \\ F_{\mu_B}(\omega_\alpha(z))-F_{\mu_\beta}(\omega_\alpha(z))\end{array} \right)\,.
\end{align}
From Lemma \ref{le lemma stability}, we know that the imaginary parts of the subordination functions $\omega_\alpha,\omega_\beta$ are uniformly bounded from below on $\mathcal{S}_{\mathcal{I}}(0,\eta_\mathrm{M})$. Next, integration by parts reveals that for any probability measures $\mu_1$ and $\mu_2$,
\begin{align}\label{le partial integration result}
 |m_{\mu_1}(z)-m_{\mu_2}(z)|\le c\frac{\dL(\mu_1,\mu_2)}{\im z}\left(1+\frac{1}{\im z}\right)\,,\qquad\qquad z\in\C^+\,,
\end{align}
with some numerical constant $c$; see, \eg \cite{Kargin}. Thus,
 \begin{align}
  |F_{\mu_A}(\omega_\beta(z))-F_{\mu_\alpha}(\omega_\beta(z))|=\frac{ |m_{\mu_A}(\omega_\beta(z)))-m_{\mu_\alpha}(\omega_\beta(z))|}{|m_{\mu_A}(\omega_\beta(z)) m_{\mu_\alpha}(\omega_\beta(z))|}\le C \frac{\dL(\mu_A,\mu_\alpha)}{(\im \omega_\beta(z))^2}\,,
 \end{align}
with a new constant $C$ that depends on the lower bound of $ \im m_{\mu_\alpha}(\omega_\beta(z)) = \im m_{\mu_\alpha\boxplus\mu_\beta}(z)$ which is strictly positive on $\mathcal{S}_{\mathcal{I}}(0,\eta_{\mathrm{M}})$; \cf~\eqref{le 44}.
 Here we used 
\begin{align*}
 \im m_{\mu_A}(\omega_\beta(z))\ge\im m_{\mu_\alpha}(\omega_\beta(z))-| \im m_{\mu_A}(\omega_\beta(z))- \im m_{\mu_\alpha}(\omega_\beta(z))|\ge\frac{1}{2} \im m_{\mu_\alpha}(\omega_\beta(z)) \,,
\end{align*}
as follows from~\eqref{le partial integration result} for small enough $\dL(\mu_A,\mu_\alpha)\le b_0$. Repeating the argument with  the r\^oles of $A$ and $B$ interchanged, we arrive at
 \begin{align} \label{080625}
  |r_A(z)| \le C\frac{\dL(\mu_A,\mu_\alpha)}{(\im \omega_\beta(z))^2}\,,\quad\qquad|r_B(z)|\le C\frac{\dL(\mu_B,\mu_\beta)}{(\im \omega_\alpha(z))^2}\,,\qquad\qquad z\in\mathcal{S}_{\mathcal{I}}(0,\eta_{\mathrm{M}})\,,
 \end{align}
 for some constant $C$. Recalling the definition of $\Gamma$ in~\eqref{le what stable means}, we get for sufficiently small $b_0$,
 \begin{align}\label{zhigang today}
\Gamma_{\mu_A,\mu_B}(\omega_\alpha,\omega_\beta)\leq 2 \Gamma_{\mu_\alpha,\mu_\beta}(\omega_\alpha,\omega_\beta)\leq 2S\,,
 \end{align}
 where $S$ is from Lemma~\ref{lemma linear stability}, and where we also use Lemma~\ref{le lemma stability} and the assumption $\dL(\mu_A,\mu_\alpha)+\dL(\mu_{B},\mu_\beta)\le b_0$. The Newton--Kantorovich theorem then implies (\cf the proof of Lemma~\ref{le lemma large eta} for a similar application) that there are $\widehat{\omega}_A(z)$, $\widehat{\omega}_B(z)$ satisfying
 \begin{align}
  \PP_{\mu_A,\mu_B}(\widehat{\omega}_A(z),\widehat{\omega}_B(z),z)=0\,,\qquad\qquad z\in\mathcal{S}_{\mathcal{I}}(0,\eta_{\mathrm{M}})\,,
 \end{align}
and
 \begin{align}
 |\omega_\alpha(z)-\widehat{\omega}_A(z)|\leq 2 \|r(z)\|\,,\qquad \qquad  |\omega_\beta(z)-\widehat{\omega}_B(z)|\leq 2 \|r(z)\|\,,\label{082050}
 \end{align}
 $z\in\mathcal{S}_{\mathcal{I}}(0,\eta_{\mathrm{M}})$. Invoking~\eqref{080625},~\eqref{082050} and Lemma~\ref{le lemma stability}, we see that $\widehat{\omega}_A(z)\in\C^+$ and $\widehat{\omega}_B(z)\in\C^+$, for any $z\in\mathcal{S}_{\mathcal{I}}(0,\eta_{\mathrm{M}})$ if $b_0$ is sufficiently small. Yet, by the global uniqueness of solutions asserted in Proposition~\ref{le prop 1}, we must have $\widehat\omega_A(z)=\omega_A(z)$, $\widehat\omega_B(z)=\omega_B(z)$, $z\in\C^+$. Together with~\eqref{082050} and  ~\eqref{080625} this  implies~\eqref{le estimate for the deterministic guys} and concludes the proof. Then, choosing $b_0$ sufficiently small,~\eqref{le upper bound on omega AB} and~\eqref{le lower bound on omega AB} are direct consequences of~\eqref{le estimate for the deterministic guys}, Lemma~\ref{le lem upper bound on omega} and Lemma~\ref{le lemma stability}.\qedhere

\end{proof}

With the aid of Lemma~\ref{lemma linear stability}, we prove the stability of the system $\PP_{\mu_A,\mu_B}(\omega_A,\omega_B,z)=0$.

\begin{corollary} \label{cor081501} Under the assumptions of Lemma~\ref{cor.080601}, there is a (small) constant $b_1>0$, depending on the measures $\mu_\alpha$ and $\mu_\beta$, on the interval $\mathcal{I}$ and on the constant $\eta_{\mathrm{M}}$, such that
\begin{align}\label{le zgs condition}
  \dL(\mu_A,\mu_\alpha)+\dL(\mu_B,\mu_\beta)\le b_1
\end{align}
implies
\begin{align}\label{le linear stability equation BIS}
\max_{z\in\mathcal{S}_{\mathcal{I}}(0,\eta_{\mathrm{M}})}\Gamma_{\mu_A,\mu_B}(\omega_A(z),\omega_B(z))\le 2 S
\end{align}
and 
\begin{align}\label{la derivative on omegaAB}
 \max_{z\in\mathcal{S}_{\mathcal{I}}(0,\eta_{\mathrm{M}})}|\omega'_A(z)|\le 4S\,,\qquad\qquad \max_{z\in\mathcal{S}_{\mathcal{I}}(0,\eta_{\mathrm{M}})}|\omega'_B(z)|\le 4S\,,
\end{align}
where $\omega_A(z)$, $\omega_B(z)$ satisfy $\PP_{\mu_A,\mu_B}(\omega_A(z),\omega_B(z),z)=0$ and $S$ is the constant in Lemma~\ref{lemma linear stability}.
\end{corollary}
\begin{proof}
Let $\Gamma\equiv \Gamma_{\mu_A,\mu_B}(\omega_A(z),\omega_B(z))$. Analogously to~\eqref{le gamma for the second}, we have
\small
\begin{align}
 \Gamma=\frac{1}{\left|1-(F_{\mu_A}'(\omega_B)-1) (F_{\mu_B}'(\omega_A)-1) \right|}\left\|\left( \begin{array}{cc}
 -1& -F_{\mu_A}'(\omega_B(z))+1  \\
 -F_{\mu_B}'(\omega_A(z))+1 & -1\\
   \end{array} \right) \right\|\,.
\end{align}\normalsize
Using the bounds~\eqref{le upper bound on omega AB} and~\eqref{le lower bound on omega AB} for sufficiently small $b_1$, we follow, mutatis mutandis, the proof of Lemma~\ref{lemma linear stability} to get~\eqref{le linear stability equation BIS}. The estimates in~\eqref{la derivative on omegaAB} then follow as in Lemma~\ref{lemma linear stability}.
  \qedhere\end{proof}
 We are now ready to complete the proof of Theorem~\ref{le theorem continuity}.

 \begin{proof}[Proof of Theorem~\ref{le theorem continuity}]
 Recall that $m_{\mu_A\boxplus \mu_B}(z)=m_{\mu_A}(\omega_B(z))$, $z\in\C^+$. We first note that 
 \begin{align*}
|m_{\mu_A}(\omega_B(z))-m_{\mu_\alpha}(\omega_B(z))|&\le C\frac{\dL(\mu_A,\mu_\alpha)}{\im \omega_\beta(z)}\left(1+\frac{1}{\im \omega_\beta(z)}\right)\,,
\end{align*}
for some numerical constant $C$; \cf~\eqref{le partial integration result}. Thus using~\eqref{le lower bound on omega AB} we get
\begin{align*}
|m_{\mu_A}(\omega_B(z))-m_{\mu_\alpha}(\omega_B(z))|\le K_2 k^{-2} \dL(\mu_A,\mu_\alpha)\,,\qquad\qquad z\in\mathcal{S}_{\mathcal{I}}(0,\eta_{\mathrm{M}})\,,
 \end{align*}
for some numerical constant $K_2$. Choosing $b$ as in Lemma~\ref{cor.080601} and assuming that $\dL(\mu_A,\mu_\alpha)+\dL(\mu_A,\mu_\beta)\le b$,  we get from~\eqref{le estimate for the deterministic guys} that
\begin{align*}
|m_{\mu_\alpha}(\omega_B(z))-m_{\mu_\alpha}(\omega_\beta(z))|\le K_1k^{-4}((\dL(\mu_A,\mu_\alpha)+\dL(\mu_A,\mu_\beta)) \,, \qquad z\in\mathcal{S}_{\mathcal{I}}(0,\eta_{\mathrm{M}})\,.
 \end{align*}
 Setting $Z\deq K_1k^{-4}+K_2k^{-2}$ we thus obtain~\eqref{le ksv statment}.
\qedhere\end{proof}
 \begin{remark}\label{la cor bound on mAB}
Note that under the assumptions of Theorem~\ref{le theorem continuity}, we have, for $\dL(\mu_A,\mu_\alpha)+\dL(\mu_A,\mu_\beta)\le b$, the bounds
 \begin{align}\label{le bound on mAB stuff}
  \kappa_1/2\le|m_{\mu_A\boxplus\mu_B}(z)|\le 1/k\,,
\end{align}
uniformly on $\mathcal{S}_{\mathcal{I}}(0,\eta_{\mathrm{M}})$ with $\kappa_1>0$ from~\eqref{le 44} and $k>0$ from~\eqref{le lower bound on omega AB}. 
 \end{remark}

\section{Proof of Theorem \ref{thm041801}}\label{le section proof of theorem matrix}
 Before we immerse into the details of the proof of Theorem~\ref{thm041801}, we outline how Theorem~\ref{thm stability} and the local stability results of Section~\ref{section perturbation} in combination with concentration estimates for the unitary groups lead to the local law in~\eqref{le local convergence of m}.

\subsection{Outline of proof} \label{s.2.4}
We briefly outline of our proof when $U$ is Haar distributed on~$U(N)$. Since we are interested in the tracial quantity $m_H$ of $H=A+UBU^*$, we may replace $H$ by the matrix
\begin{align}
 \widetilde H\deq VAV^*+UBU^*\,,
\end{align}
where $V$ is another Haar unitary independent from~$U$. By cyclicity of the trace we have~$m_H=m_{\widetilde H}$ and we study $m_{\widetilde H}$ below. We emphasize that this replacement is a convenient technicality which is not essential to our proof.

Using the shorthand
\begin{align}
\widetilde{A}\deq VAV^*\,,\quad\qquad \widetilde{B}\deq UBU^* \,,\label{071905}
\end{align}
we introduce the Green functions 
\begin{eqnarray}\label{le greens functions}
G_{\widetilde{A}}(z)\deq(\widetilde{A}-z)^{-1},\qquad G_{\widetilde{B}}(z)\deq(\widetilde{B}-z)^{-1}\,,\qquad\qquad z\in\C^+\,.
\end{eqnarray}
For a given $N\times N$ matrix~$Q$, we introduce the function
\begin{align}\label{le ff}
f_Q(z)\deq\ntr QG_{\widetilde H}(z)\,,\qquad\qquad z\in\C^+\,,
\end{align}
where $G_{\widetilde H}=(\widetilde H-z)^{-1}$ is the Green function of ${\widetilde H}$. We define  the {\it approximate subordination functions}, $\omega_A^c$ and $\omega_B^c$, by setting
\begin{align}\label{060104}
\omega_A^c(z)\deq z-\frac{\mathbb{E} f_{\widetilde{A}}(z)}{\mathbb{E} m_{\widetilde H}(z)}\,,\qquad \omega_B^c(z)\deq z-\frac{\mathbb{E} f_{\widetilde{B}}(z)}{\mathbb{E} m_{\widetilde H}(z)}\,,\quad\qquad 
 z\in\C^+\,,
\end{align}
where the expectation $\mathbb{E}$ is with respect to both Haar unitaries $U$ and $V$. From
the identity $({\widetilde H}-z)G_{\widetilde H}(z)=1$, $z\in\C^+$, we then obtain the relation
\begin{align}\label{le third equation}	
\omega_A^c(z)+\omega_B^c(z)-z=-\frac{1}{\E m_{\widetilde H}(z)}\,,\qquad\qquad z\in\C^+\,,
\end{align}
reminiscent to  (\cf \eqref{le definiting equations}--\eqref{le kkv})
$$
    \omega_A (z)+\omega_B (z)-z=-\frac{1}{ m_{A\boxplus B}(z)}\,,\qquad\qquad z\in\C^+\,.
$$
For the proof of Theorem~\ref{thm041801}, we decompose
\begin{eqnarray}
m_{\widetilde H}(z)-m_{A\boxplus B}(z)=\big(m_{\widetilde H}(z)-\mathbb{E}m_{\widetilde H}(z)\big)+\big(\mathbb{E}m_{\widetilde H}(z)-m_{A\boxplus B}(z)\big)\,, \label{071720}
\end{eqnarray}
where we abbreviate $m_{A \boxplus B}\equiv m_{\mu_{A}\boxplus\mu_{B}}$. To control the fluctuation part, $m_{\widetilde H}(z)-\mathbb{E}m_{\widetilde H}(z)$, we rely on the Gromov--Milman concentration inequality~\cite{GM83} for the unitary group; see~\eqref{070501} below. To control the deterministic part, we first note that, by~\eqref{le third equation} and $m_{A\boxplus B}(z)=m_{A}(\omega_B((z))$, bounding $|\mathbb{E}m_{\widetilde H}(z)-m_{A\boxplus B}(z)|$ amounts to bounding $|\omega_A^c(z)-\omega_A(z)|$ and $|\omega_B^c(z)-\omega_B(z)|$. We then show that $\omega_A^c(z)$ and $\omega_B^c(z)$ are both in the upper-half plane and satisfy
\begin{align}
\PP_{\mu_A,\mu_B}(\omega_A^c(z),\omega_B^c(z),z)=r(z)\,,\qquad\qquad\quad\; z\in \mathcal{S}_{\mathcal{I}}(\eta_{\mathrm{m}}, \eta_{\mathrm{M}})\,, \label{081201}
\end{align}
for some small error $r(z)\in\C^+$, \ie we consider~\eqref{081201} as a perturbation of the system~$
\PP_{\mu_A,\mu_B}(\omega_A(z),\omega_B(z),z)=0$; 
\cf~\eqref{le H system defs}.
 The formal derivation of~\eqref{081201} goes back to Pastur and Vasilchuk~\cite{VP}. Using Proposition~\ref{le proposition perturbation of system} (with rough a priori estimates on $|\omega_A^c(z)-\omega_A(z)|$ and $|\omega_B^c(z)-\omega_B(z)|$ obtained from the continuity argument below) and stability results of Theorem~\ref{thm stability} and of Section~\ref{section theorem continuity}, we then bound $|\omega_A^c(z)-\omega_A(z)|$ and $|\omega_B^c(z)-\omega_B(z)|$ in terms of $r(z)$.

In sum, for fixed $z\in\C^+$, our proof includes two parts: $(i)$ estimation of the error~$r(z)$ in~\eqref{081201} and $(ii)$ concentration for $m_{\widetilde H}(z)$ around $\mathbb{E}m_{\widetilde H}(z)$. Both parts
rely on the estimates
\begin{align}
\mathbb{E}m_{\widetilde H}(z)\,,\, \omega_A^c(z)\,,\, \omega_B^c(z)\sim 1\,,\quad \Im\omega^c_A(z)\,,\, \Im \omega^c_B(z)\gtrsim 1\,,\qquad z\in \mathcal{S}_{\mathcal{I}}(\eta_{\mathrm{m}}, \eta_{\mathrm{M}})\,. \label{081205}
\end{align}
Note that the quantities in~\eqref{081205} are obtained from the Green function of ${\widetilde H}$ by averaging with respect to the Haar measure. Similar bounds for $m_{A\boxplus B}$, $\omega_A$ and $\omega_B$ were obtained in Section~\ref{section theorem continuity}. These latter quantities are defined directly from~$\mu_A$ and~$\mu_B$ via~Proposition~\ref{le prop 1}

To establish~\eqref{081205}, we use a similar continuity argument as was used for Wigner matrices in~\cite{EYY}: For $\im z=\eta_{\mathrm{M}}$ sufficiently large, the estimates in ~\eqref{081205} directly follow from definitions. For $z=E+\ii\eta$, with $E\in \mathcal{I}$ fixed, we decrease $\eta=\eta_{\mathrm{M}}$ down to $\eta=\eta_{\mathrm{m}}$ in steps of size $O(N^{-5})$, where, at each step, we invoke parts~$(i)$ and $(ii)$. However, a direct application of the  Gromov--Milman concentration inequality for part $(i)$ does not allow to push $\eta$ below the mesoscopic scale $\eta= N^{ -1/2}$. Indeed, the Gromov--Milman inequality is effective if $\mathcal{L}^{2}/N=o(1)$, where~$\mathcal{L}$ is the Lipschitz constant of $m_{{\widetilde H}}(z)$ with respect to the Haar unitary $V$. It is roughly bounded by $\sqrt{\ntr |G_{\widetilde H}(z)|^4/N}$, which in turn is trivially bounded by $1/\sqrt{N\eta^4}$, giving the $\eta\ge N^{-1/2+\gamma}$, $\gamma>0$, threshold. However, in reality, the random quantity
$\sqrt{\ntr |G_{\widetilde H}(z)|^4/N}$ is typically of order $1/\sqrt{N\eta^3}$ as follows by combining the deterministic estimate $\ntr |G_{\widetilde H}(z)|^4\leq  \eta^{-3}\Im m_{\widetilde H}(z)$ with a probabilistic order one bound for~$\im m_{\widetilde H}(z)$. 

Our key novelty here is to capitalize on this latter information. We introduce a smooth cutoff that regularizes $m_{\widetilde H}(z)$ and then apply the Gromov--Milman inequality for this regularized quantity. With the bound  $1/\sqrt{N\eta^3}$ for the Lipschitz constant, we get concentration estimates down to scales~$\eta\ge N^{-2/3+\gamma}$, $\gamma>0$. 

\subsubsection{Notation}
The following notation for high-probability estimates is suited for our purposes. A slightly different form was first used in~\cite{EKY}.
\begin{definition}\label{definition of stochastic domination}
Let
\begin{align}
 X=(X^{(N)}(v)\,:\, N\in\N\,, v\in \mathcal{V}^{(N)})\,,\qquad\qquad Y=(Y^{(N)}(v)\,:\, N\in\N\,,\,v\in \mathcal{V}^{(N)})
\end{align}
be two families of nonnegative random variables where $\mathcal{V}^{(N)}$ is a possibly $N$-dependent parameter set. We say that $Y$ stochastically dominates $X$, uniformly in~$v$, if for all (small) $\epsilon>0$ and (large) $D>0$,
\begin{align}
 \P\,\bigg(\bigcup_{v\in \mathcal{V}^{(N)}}\bigg\{X^{(N)}(v)>N^{\epsilon} Y^{(N)}(v) \bigg\}\bigg)\le N^{-D}\,, \label{080630}
\end{align}
for sufficiently large $N\ge N_0(\epsilon,D)$. If $Y$ stochastically dominates $X$, uniformly in~$v$, we write $X \prec Y$. 
If we wish to indicate the set $\mathcal{V}^{(N)}$ explicitly, we write that
$X(v) \prec Y(v)$ for all $v\in \mathcal{V}^{(N)}$.
\end{definition}

\subsection{Localized Gromov--Milman concentration estimate}\label{la concentration for the random part}
 In this subsection, we derive concentration bounds for some key tracial quantities. They are tailored for the continuity argument of Subsection~\ref{la boostrapping subsection} used to complete the proof of Theorem~\ref{thm041801}. The argument works with $U$, $V$ independent and both Haar distributed on $U(N)$ or on $O(N)$. Below, $\mathbb{E}$ denotes the expectation with respect Haar measure.

 In the rest of this section, we let $\mathcal{I}\subset\mathcal{B}_{\mu_\alpha\boxplus\mu_\beta}$ denote the compact non-empty subset fixed in Theorem~\ref{thm041801}. Also recall from Theorem~\ref{thm041801} that we set $\eta_m=N^{-2/3+\gamma}$, $\gamma>0$. Below we choose the constant $\eta_M\sim 1$ to be sufficiently large at first, but from the proof it will be clear that we can eventually choose $0<\eta_{\mathrm{M}}<\infty$ arbitrary. Recall from~\eqref{le ff} the notation~$f_Q$, where~$Q$ is an arbitrary $N\times N$ matrix.

\begin{proposition}\label{prop041901}
 Let $Q$ be a given $N\times N$ deterministic matrix with $\|Q\|\lesssim 1$. Fix $E\in \mathcal{I}$ and $\widehat\eta\in[\eta_{\mathrm{m}},\eta_{\mathrm{M}}]$. Then
\begin{eqnarray}
\Im m_{\widetilde H}(E+\ii \eta)\prec 1\,,\qquad \qquad\forall  \eta\in[\widehat{\eta}, \eta_{\mathrm{M}}]\,, \label{041816}
\end{eqnarray} 
implies the concentration bound
\begin{align}
|f_{VQV^*}(E+\ii\widehat{\eta})-\mathbb{E} f_{VQV^*}(E+\ii\widehat{\eta})|\prec \frac{1}{\sqrt{N^2\widehat{\eta}^{3}}}\,. \label{041815}
\end{align}
The same concentration holds with $VQV^*$ replaced by $UQU^*$.
\end{proposition}
\begin{proof} For fixed $E\in \mathcal{I}$, we consider $z=E+\ii\eta\in\C^+$ as a varying spectral parameter and use $\widehat{z}=E+\ii\widehat{\eta}$ for the specific choice in the lemma. By the definition of $f_{(\cdot)}$ and cyclicity of the trace, we have
\begin{align}
f_{VQV^*}(z)=\ntr VQV^*\big(VAV^*+UBU^*-z\big)^{-1}= \ntr Q\big(A+V^*UBU^*V-z\big)^{-1}\,, \label{082201}
\end{align}
where $\ntr(\,\cdot\,)$ stands for the normalized trace. For simplicity, we denote
\begin{align}
W\deq V^*U,\qquad \mathcal{H}\deq A+WBW^*\,, \qquad G_{\mathcal{H}}(z)\deq(\mathcal{H}-z)^{-1}\,. \label{071910}
\end{align}
Observe that $W$ is Haar distributed on $U(N)$, respectively $O(N)$, too. By cyclicity of the trace we have $\ntr G_\mathcal{H}(z)=\ntr G_H(z)=m_H(z)$.
According to~\eqref{082201} and~\eqref{071910}, we may regard in the sequel $f_{VQV^*}$ as a function of the Haar unitary matrix~$W$ by writing
 \begin{align*}
 h(z)=h_W(z)  \deq f_{VQV^*}(z)\,.
 \end{align*}
 
For any fixed (small) $\varepsilon>0$, let $\widehat\chi$ be a smooth cutoff supported on $[0,2N^{\varepsilon}]$, with $\widehat\chi(x)=1$, $x\in [0,N^{\varepsilon}]$, and with bounded derivatives. Since $m_H(z)=\ntr G_\mathcal{H}(z)$, we can regard~$m_H(z)$ as a function of~$W$ and write
\begin{align}\label{070110}
\chi(z)= \chi_W(z)\deq\widehat\chi(\Im m_H(z))\,.
\end{align}
We then introduce a regularization, $\widetilde{h}_W$, of $h_W$ by setting
\begin{align}
\widetilde{h}(z)=\widetilde{h}_{W}(z)\deq h_{W}(z)\prod_{n=0}^{\lceil -\log_2 \eta\rceil} \chi_W(E+\ii2^{n}\eta)\,.  \label{041820}
\end{align}
We will often drop the $W$ subscript from the notations $h_W(z)$, $\widetilde{h}_{W}(z)$
 and $\chi_W(z)$ but remember that these
are random variables depending on the Haar unitary $W$.

We will use assumption~\eqref{041816} at dyadic points, \ie  that
\begin{align}
\Im m_{H}(E+\ii2^l\widehat{\eta})\prec 1\,,\qquad\quad 0\leq l\leq \lceil -\log_2 \widehat\eta\rceil\,, \label{070102}
\end{align}
(recall that $m_H(z) = m_{\widetilde H}(z)$ so we may drop the tilde in the subscript of $m$).
Hence, by~\eqref{070110} and~\eqref{070102} we see that, for arbitrary large $D>0$,
\begin{align}
\prod_{l=0}^{\lceil -\log_2 \widehat{\eta}\rceil} \chi(E+\ii2^{l}\widehat{\eta})=1\,, \qquad 
\ie\qquad\widetilde{h}(\widehat{z})=h(\widehat{z})\,, \label{070111}
\end{align}
with probability larger than $1-N^{-D}$, for $N$ sufficiently large (depending on $\varepsilon$ and $D$). Taking the trivial bound $\|Q\|/\widehat\eta$ for $h(\widehat z)$ and for $\widetilde{h}(\widehat z)$ into account, we also have
\begin{align}
\mathbb{E}\,\widetilde{h}(\widehat{z})-\mathbb{E}\,h(\widehat{z})=O\left(N^{-D+1}\right)\,. \label{070112}
\end{align}
To prove~\eqref{041815}, it therefore suffices to establish the concentration estimate
\begin{align}
\left|\widetilde{h}(\widehat{z})-\mathbb{E}\widetilde{h}(\widehat{z})\right|\prec \frac{1}{\sqrt{N^2\widehat{\eta}^3}}\,, \label{070115}
\end{align}
for the regularized quantity $\widetilde h(\widehat z)$.

To verify~\eqref{070115}, we use the Gromov--Milman concentration inequality~\cite{GM83} (see Theorem~4.4.27 in~\cite{AGZ} for similar applications) which states the following. Let $\mathcal{M}(N)=\mathrm{SO}(N)$ or $\mathrm{SU}(N)$ endowed with the Riemann metric $\|\dd s\|_2$ inherited from $M_N(\C)$ (equipped with the Hilbert-Schmidt norm). If $g\,:\, (\mathcal{M}(N), \|{\rm d}s\|_2)\to \mathbb{R}$ is an $\mathcal{L}$-Lipschitz function satisfying $\mathbb{E}g=0$, then
\begin{eqnarray}
\mathbb{P}\left(|g|>\delta\right)\leq\e{-c\frac{N\delta^2}{{\mathcal{L}^2}}}\, ,\qquad\quad \forall \, \delta>0\,, \label{070501}
\end{eqnarray}
with some numerical constant $c>0$ (depending only on the symmetry type and not on $N$). Here $\mathbb{P}$ and $\mathbb{E}$ are with respect Haar measure on $\mathcal{M}(N)$.

In order to apply~\eqref{070501} to the function $W\mapsto\widetilde{h}_W(\widehat{z}) = \widetilde{h}(\widehat{z})$, we need to control its Lipschitz constant. To that end, we define the event  
\begin{align}
\Omega(\widehat{\eta})\equiv \Omega_E(\widehat{\eta})\deq \big\{ \Im m_H (E+\ii2^{n}\widehat{\eta})\leq 2N^\varepsilon \; : \; \forall n\in \N_0 \big\}\,.\label{071101}
\end{align} 
To bound the Lipschitz constant, we need to bound quantities of the form $\ntr |G_{\mathcal{H}}(\widehat{z})|^k$ restricted to the event $\Omega(\widehat{\eta})$. Let $(\lambda_i(\mathcal{H}))$ denote the eigenvalues of $\mathcal{H}$ and introduce
\begin{eqnarray*}
\mathcal{I}_n\deq[E-2^n\widehat{\eta}, E+2^n \widehat{\eta}]\cap\mathcal{I}\,,\qquad\qquad N_n\deq|\{i:\lambda_i(\mathcal{H})\in \mathcal{I}_n\}|\,,\qquad n\in\N_0\,.
\end{eqnarray*}
Since $H$ and $\mathcal{H}$ are unitarily equivalent, their
 empirical eigenvalue distributions are the same,~$\mu_H$;  \cf~\eqref{le empirical distribution lambdas}. Using the definition of the Stieltjes transform we have, for all $n\in\N_0$, the estimate
\begin{align*}
 N_n=N\int_{\mathcal{I}_n}\dd \mu_{H}\le\frac{3}{2}N\cdot 2^{n+1}\widehat\eta\int_{E-2^n\widehat{\eta}}^{E+2^n\widehat{\eta}}\frac{\widehat\eta\,\dd\mu_{ H}(x)}{(x-E)^2+\widehat\eta^2}\le 3N\cdot 2^n\widehat\eta\,\im m_{H}(\widehat z)\,.
\end{align*}
Thus we have on the event~$\Omega(\widehat{\eta})$ that
\begin{eqnarray}
N_n\lesssim 2^{n}N^{1+\varepsilon}\widehat{\eta}\,,\qquad \qquad\forall n\in\N_0\,. \label{070103}
\end{eqnarray}
By the spectral theorem, we can bound
\begin{align}\label{schnabel}
  \ntr |G_{\mathcal{H}}(\widehat{z})|\lesssim \frac{1}{N}\sum_{i=1}^N\frac{1}{|\lambda_i(\mathcal{H})-E|+\widehat{\eta}}\,.
\end{align}
Then we observe (with the convention $\mathcal{I}_{-1}=\emptyset$) that
\begin{align*}
  \frac{1}{N}\sum_{i=1}^N\frac{1}{|\lambda_i(\mathcal{H})-E|+\widehat{\eta}}&= \frac{1}{N}\sum_{i=1}^N\sum_{n=0}^\infty {\bf 1}(\lambda_i\in \mathcal{I}_{n}\setminus\mathcal{I}_{n-1} )\frac{1}{|\lambda_i(\mathcal{H})-E|+\widehat{\eta}}\nonumber\\
  &= \frac{1}{N}\sum_{n=0}^{\lceil c\log N\rceil}\sum_{\lambda_i\in\mathcal{I}_{n}\setminus\mathcal{I}_{n-1} } \frac{1}{|\lambda_i(\mathcal{H})-E|+\widehat{\eta}}\,,
\end{align*}
where we used $\|\mathcal{H}\|\le C$ to truncate the sum over $n$ at $\lceil c\log N\rceil$. We then bound
\begin{align*}
 {\bf 1}(\Omega(\widehat{\eta}))\frac{1}{N}\sum_{n=0}^{\lceil c\log N\rceil}\sum_{\lambda_i\in\mathcal{I}_{n}\backslash\mathcal{I}_{n-1} } \frac{1}{|\lambda_i(\mathcal{H})-E|+\widehat{\eta}}\le{\bf 1}(\Omega(\widehat{\eta}))\frac{1}{N}\sum_{n=0}^{\lceil c\log N\rceil}\frac{N_n}{2^n\widehat\eta} \lesssim N^{\varepsilon}\log N\,,
\end{align*}
where we used~\eqref{070103}, \ie with~\eqref{schnabel} we arrive at
\begin{align}
  {\bf 1}(\Omega(\widehat{\eta}))\ntr |G_{\mathcal{H}}(\widehat{z})|\lesssim  N^{\varepsilon}\log N\,. \label{041835}
\end{align}
Using the spectral decomposition of $\mathcal{H}$ we see that
\begin{align}\label{le ward}
\ntr |G_{\mathcal{H}}(\widehat z)|^2=\frac{1}{N}\sum_{i=1}^N\frac{1}{|\lambda_i(\mathcal{H})-E|^2+\widehat\eta^2}=\frac{1}{N\widehat\eta}\sum_{i=1}^N\frac{\widehat\eta}{|\lambda_i(\mathcal{H})-E|^2+\widehat\eta^2}=\frac{\im m_H(\widehat z)}{\widehat\eta}\,,
\end{align}
where we also used that $\ntr G_\mathcal{H}(\widehat z)=\ntr G_H(\widehat z)=m_H(\widehat z)$. Thus, we bound
\begin{align}
{\bf 1}(\Omega(\widehat{\eta}))\ntr |G_{\mathcal{H}}(\widehat{z})|^k \leq{\bf 1}(\Omega(\widehat{\eta}))\widehat{\eta}^{-k+1} \Im m_H(\widehat{z})\lesssim N^\varepsilon\widehat{\eta}^{-k+1}\,,\quad\qquad \forall k\geq 2\,. \label{070502}
\end{align}

Having established~\eqref{041835} and~\eqref{070502}, we proceed to estimate the Lipschitz constant of~$\widetilde{h}(\widehat{z})$ as a function of $W$. Let $\mathfrak{su}(N)$ and $\mathfrak{so}(N)$ denote the (fundamental representations in $M_N(\C)$ of the) Lie algebras of $SU(N)$ and $SO(N)$ respectively. Let $\mathfrak{m}$ stand for either $\mathfrak{su}(N)$ or $\mathfrak{so}(N)$. Note that $X\in\mathfrak{m}$ satisfies $X^*=-X$. Since $SU(N)$ and $SO(N)$ are matrix groups the Lie bracket of $\mathfrak{su}(N)$ and $\mathfrak{so}(N)$ respectively is given by the commutator in the matrix algebras. For fixed $X\in M_N(\C)$, we let $\mathrm{ad}_X\,:\, M_N(\C)\to M_N(\C)$, $Y\mapsto\mathrm{ad}_X(Y)\deq XY-YX$. For $X\in \mathfrak{m}$ and $t\in\R$, we may write $\e{t \mathrm{ad}_X}(WBW^*)=(\e{t X}W)B(\e{t X}W)^*$, where we used that $X^*=-X$. Further note that 
\begin{align}\label{le adjoint derivata}
 \frac{\dd}{\dd t}\e{t \mathrm{ad}_X}(WBW^*)=\e{t \mathrm{ad}_X}\mathrm{ad}_X(WBW^*)\,.
\end{align}
For $X\in\mathfrak{m}$ with $\|X\|_2=1$, we let
\begin{eqnarray*}
G_{\mathcal{H}}(z,tX)\deq\Big(A+\e{t\mathrm{ad}_X}(WBW^*)-z\Big)^{-1}\,,\qquad \qquad t\in\R\,,
\end{eqnarray*}
and denote accordingly
\begin{align}
m_H(z,tX)\deq\ntr G_{\mathcal{H}}(z,tX)\,,\qquad \qquad
\chi(z,tX)\deq\widehat\chi(\Im m_H(z,tX))\,,\nonumber\\
\widetilde{h}(z,tX)\deq\ntr QG_{\mathcal{H}}(z,tX)\small\prod_{l=0}^{\lceil -\log_2 \widehat{\eta}\rceil}\normalsize \chi(E+\ii2^{l}\eta,tX)\,,\nonumber
\end{align} 
with $\chi(z,0)\equiv \chi(z)$, $\widetilde h(z, 0)\equiv\widetilde h(0)$, {\it etc}. 
{
    \setlength{\lineskiplimit}{1pt}
    \setlength{\lineskip}{1pt} 
  \setlength{\abovedisplayskip}{2pt}
   \setlength{\belowdisplayskip}{1pt}
   \setlength{\abovedisplayshortskip}{1pt}
   \setlength{\belowdisplayshortskip}{1pt}
   \noindent

Evaluating the derivative of~$\widetilde{h}(\widehat{z},tX)$ with respect to $t$ at $t=0$ we get
\begin{align}
\frac{\partial}{\partial t}\widetilde{h}(\widehat{z},tX)\Big{|}_{t=0}=&\,-\ntr\! \Big(Q G_{\mathcal{H}}(\widehat{z})\mathrm{ad}_X(WBW^*)G_{\mathcal{H}}(\widehat{z})\Big)\prod_{l=0}^{\lceil -\log_2 \widehat{\eta}\rceil} \chi(E+\ii2^{l}\widehat{\eta})\nonumber\\
&-\ntr\!\big(QG_{\mathcal{H}}(\widehat{z})\big)\bigg(\sum_{j=0}^{\lceil -\log_2 \widehat{\eta}\rceil}\Big[\prod_{\substack{l=0 \\ l\neq j}}^{\lceil -\log_2 \widehat{\eta}\rceil}\chi(E+\ii2^{l}\widehat{\eta})\Big]\cdot \phi(E+\ii2^j\widehat{\eta}) \nonumber\\
&\qquad\quad\times  \Im \ntr\!\Big( G_{\mathcal{H}}(E+\ii2^j\widehat{\eta})\,\mathrm{ad}_X(WBW^*)G_{\mathcal{H}}(E+\ii2^j\widehat{\eta})\Big)\bigg)\,, \label{070505}
\end{align}
where we used~\eqref{le adjoint derivata} and where we introduced $\phi(z)\deq \widehat\chi'(\im m_H(z))$, with $\widehat\chi'$ the derivative of $\widehat\chi$. Recalling~\eqref{070110} and the definition of the cutoff $\widehat\chi$, we note the bounds
\begin{align}
\small\prod_{l=0}^{\lceil -\log_2 \widehat{\eta}\rceil}\normalsize\chi(E+\ii2^{l}\widehat{\eta})\leq 1\,,\qquad \small\sum_{j=0}^{\lceil -\log_2 \widehat{\eta}\rceil}\normalsize\Big[\prod_{l \neq j}\chi(E+\ii2^{l}\widehat{\eta})\Big]\cdot \phi(E+\ii2^j\widehat{\eta})= O(\log N)\,. \label{070503}
\end{align}
On the event $\Omega^{\mathrm{c}}(\widehat{\eta})$, the complementary event to $\Omega(\widehat{\eta})$, we further have the identities
\begin{align}
\small\prod_{l=0}^{\lceil -\log_2 \widehat{\eta}\rceil}\normalsize \chi(E+\ii2^{l}\widehat{\eta})=0\,,&\quad \small\sum_{j=0}^{\lceil -\log_2 \widehat{\eta}\rceil}\normalsize\Big[\prod_{\substack{l=0 \\ l\neq j}}^{\lceil -\log_2 \widehat{\eta}\rceil}\chi(E+\ii2^{l}\widehat{\eta})\Big]\cdot \phi(E+\ii2^j\widehat{\eta})=0\,.\label{070504}
\end{align}
It thus suffices to bound~\eqref{070505} on the event $\Omega(\widehat{\eta})$. We bound the first term on the right side of~\eqref{070505} as

\begin{align}\label{le junge r}
&{\bf 1}(\Omega(\widehat{\eta}))\Big|\ntr\!\Big( Q G_{\mathcal{H}}(\widehat{z})\mathrm{ad}_X(WBW^*)\,G_{\mathcal{H}}(\widehat{z})\Big)\prod_{l=0}^{\lceil -\log_2 \widehat{\eta}\rceil} \chi(E+\ii2^{l}\widehat{\eta})\Big| \nonumber\\
&\qquad\le {\bf 1}(\Omega(\widehat{\eta}))\frac{1}{N}\|\mathrm{ad}_X(WBW^*) \|_2\|G_{\mathcal{H}}(\widehat{z})Q G_{\mathcal{H}}(\widehat{z}) \|_2 \,\prod_{l=0}^{\lceil -\log_2 \widehat{\eta}\rceil} \chi(E+\ii2^{l}\widehat{\eta})\,,
\end{align}
where we used cyclicity of the trace and Cauchy--Schwarz inequality.
Next, note that  $\|\mathrm{ad}_X(WBW^*) \|_2\le2 \|B\|\|X\|_2\le 2\|B\|$, where we used the definition of $\mathrm{ad}_X$, $\|W\|\le 1$ and $\|X\|_2=1$. Similarly, we have $\|G_{\mathcal{H}}(\widehat{z})Q G_{\mathcal{H}}(\widehat{z}) \|_2 \le \|Q\| \|G_{\mathcal{H}}(\widehat{z})G_{\mathcal{H}}^*(\widehat{z})  \|_2$. Thus from~\eqref{le junge r},
\begin{align}\label{why not give it a name too}
{\bf 1}(\Omega(\widehat{\eta}))\Big|\ntr\!\Big( Q G_{\mathcal{H}}(\widehat{z})&\mathrm{ad}_X(WBW^*)\,G_{\mathcal{H}}(\widehat{z})\Big)\prod_{l=0}^{\lceil -\log_2 \widehat{\eta}\rceil} \chi(E+\ii2^{l}\widehat{\eta})\Big| \nonumber\\
&\le 2\|B\|\|Q\|{\bf 1}(\Omega(\widehat{\eta})) \left({\frac{\ntr |G_{\mathcal{H}}(\widehat{z})|^4}{N}}\right)^{\frac12}\prod_{l=0}^{\lceil -\log_2 \widehat{\eta}\rceil} \chi(E+\ii2^{l}\widehat{\eta})\nonumber\\
&\lesssim \left({\frac{N^\varepsilon}{N\widehat{\eta}^3}}\right)^{\frac12}\,,
\end{align}
where we used~\eqref{070502} with $k=4$ in the last step.

To handle the second term on the right side of~\eqref{070505}, we use~\eqref{070503} and~\eqref{041835} to get 
\begin{align}
{\bf 1}(\Omega(\widehat{\eta}))\Big| \ntr\!\big( QG_{\mathcal{H}}(\widehat{z})\big)&\sum_{j=0}^{\lceil -\log_2 \widehat{\eta}\rceil}\prod_{\substack{l=0 \\ l\neq j}}^{\lceil -\log_2 \widehat{\eta}\rceil}\chi(E+\ii2^{l}\widehat{\eta})\cdot \phi(E+\ii2^j\widehat{\eta}) \nonumber\\
&\qquad\qquad\times  \Im \ntr\!\Big( G_{\mathcal{H}}(E+\ii2^j\widehat{\eta})\mathrm{ad}_X(WBW^*)\,G_{\mathcal{H}}(E+\ii2^j\widehat{\eta})\Big)\Big|\nonumber\\
&\leq{\bf 1}(\Omega(\widehat{\eta})) {\|Q\|}\ntr |G_{\mathcal{H}}(\widehat{z})|\; \sum_{j=0}^{\lceil -\log_2 \widehat{\eta}\rceil}\prod_{\substack{l=0 \\ l\neq j}}^{\lceil -\log_2 \widehat{\eta}\rceil}\chi(E+\ii2^{l}\widehat{\eta})\; |\phi(E+\ii2^{l}\widehat{\eta})|\;\nonumber\\
&\qquad\qquad\times{2\|B\| \|Q\|} \left({\frac{\ntr |G_{\mathcal{H}}(E+\ii2^{j}\widehat{\eta})|^4}{N}}\right)^{\frac12} \nonumber\\
&\lesssim \left({\frac{N^{4\varepsilon}}{N\widehat{\eta}^3}}\right)^{\frac12}\,.\label{why not give it a name}
\end{align}

Combining~\eqref{why not give it a name} and~\eqref{why not give it a name too} we obtain, for any $X\in\mathfrak{su}(N)$ or $\mathfrak{so}(N)$ with $\|X\|_2=1$, that 
\begin{align}
\left|\frac{\partial}{\partial t}\widetilde{h}(\widehat{z},tX)\Big{|}_{t=0}\right|\lesssim \left({\frac{N^{4\varepsilon}}{N\widehat{\eta}^3}}\right)^{\frac12}\,,\label{072330}
\end{align} 
\ie the Lipschitz constant of $\widetilde h(\widehat{z})$ as a function of $W$ is bounded by $C \left({\frac{N^{4\varepsilon}}{N\widehat{\eta}^3}}\right)^{1/2}$, for some constant $C$ depending only on $\|B\|$ and $\|Q\|$. Thus, taking
\begin{gather}
g=\widetilde{h}(z)-\mathbb{E}\widetilde{h}(z)\,,\qquad \mathcal{L}=C \left(\frac{N^{4\varepsilon}}{N\widehat{\eta}^3}\right)^{\frac{1}{2}}\,,\qquad\delta=N^{3\varepsilon}/\sqrt{N^2\widehat{\eta}^{3}}\nonumber
\end{gather}
in~\eqref{070501}, and choosing $\varepsilon>0$ sufficiently small, we get~\eqref{070115}. Together with~\eqref{070111} and~\eqref{070112} this implies~\eqref{041815}.
}\qedhere\end{proof}

\subsection{Continuity argument}\label{la boostrapping subsection} In this subsection, we often omit $z\in\C^+$ from the notation. Let $U$ and $V$ be independent and both Haar distributed on either $U(N)$ or $O(N)$.
Recalling the notation in Section \ref{s.2.4}, we set
\begin{align}
&\Delta_A(z)\deq-(\IE[m_H(z)]) G_{\widetilde H}(z)-(\IE[f_{\widetilde{B}}(z)])G_{\widetilde{A}}(z)G_{\widetilde H}(z)\,,\nonumber\\
&\Delta_B(z)\deq-(\IE[m_H(z)])G_{\widetilde H}(z)-(\IE[f_{\widetilde{A}}(z)])G_{\widetilde{B}}(z)G_{\widetilde H}(z)\,, \qquad z\in\C^+\,,\label{041905}
\end{align}
where we introduced $\IE X\deq X-\E X$, for any random variables $X$. Using the left-invariance of Haar measure, one derives the identities
\begin{align*}
\mathbb{E}[ G_{\widetilde H}\otimes \widetilde{A}G_{\widetilde H}]=\mathbb{E} [\widetilde{A}G_{\widetilde H}\otimes G_{\widetilde H}]\,,\qquad \mathbb{E}[ G_{\widetilde H}\otimes \widetilde{B}G_{\widetilde H}]=\mathbb{E}[ \widetilde{B}G_{\widetilde H}\otimes G_{\widetilde H}]\,;
\end{align*}
see Theorem~7~in~\cite{VP} or Appendix~A of~\cite{Kargin} for proofs. Taking the partial trace for the first component of the tensor products, we get
\begin{align}\label{041301}
 \mathbb{E}G_{\widetilde H}(z)=\mathbb{E}G_{\widetilde{A}}(\omega_B^c(z))+\delta_A^c(z)\,,
 \qquad \delta_A^c(z)\deq\frac{1}{\mathbb{E}m_H(z)} \mathbb{E} [G_{\widetilde{A}}(\omega_B^c(z))(\widetilde{A}-z)\Delta_A(z)]\,,
\end{align}
 where $\omega_B^c(z)$ is defined in~\eqref{060104}, we used \eqref{le third equation},
 and where we implicitly assumed that $\im \omega_B^c(z)>0$. This last assumption will be verified along the continuity argument. Then, we set
\begin{align} \label{071508}
&r_A^c(z)\deq-\frac{\ntr \delta_A^c(z)}{\ntr G_{\widetilde A}(\omega_B^c(z))(\ntr G_{\widetilde{A}}(\omega_B^c(z))+\ntr \delta_A^c(z))}\,,
\end{align}
and define $\delta_B^c(z)$ and $r_B^{c}(z)$ in the same way by swapping the r\^oles of $A$ and $B$.
Using~\eqref{041301},~\eqref{le third equation},  we eventually obtain, under the assumption that $\im \omega_A^c(z)>0$ and $\im\omega_B^c(z)>0$,
\begin{align}
\PP_{\mu_A,\mu_B}(\omega_A^c(z),\omega_B^c(z),z)=r^c(z)\,,\qquad\qquad z\in\C^+\,, \label{071517}
\end{align}
with $r^c(z)=(r_A^c(z),r_B^c(z))^\top$.

\begin{lemma} \label{lem071701}
Fix $E\in \mathcal{I}$ and any $\wh\eta \in [\eta_\mathrm{m},\eta_{\mathrm{M}}]$.
Set the notation $z=E+\ii \eta$ and $\widehat{z}=E+\ii \widehat{\eta}$. Suppose that 

\begin{align}
|\omega_A^c(z)-\omega_A(z)|+|\omega_B^c(z)-\omega_B(z)|\le N^{-\gamma}\, ,\qquad
\forall \eta=\im z \in [\widehat{\eta},\eta_{\mathrm{M}}]\,.
\label{071502}
\end{align} 
 Moreover, assume that for the event
\begin{align*}
 \Xi(\wh\eta)\equiv\Xi_E(\wh\eta)  \deq \Big\{ |m_H(z)-m_{A\boxplus B}(z)| \le N^{-\gamma} \; : \;  z=E+\ii \eta, \quad
   \forall\eta\in [\widehat{\eta},\eta_{\mathrm{M}}] \Big\}
   \end{align*}
we have
\begin{equation}\label{Xi}
   \P\big(\Xi(\wh\eta)\big) \ge 1-N^{-D}\big(1+N^5 (\eta_{\mathrm{M}}-\wh\eta)\big)\,,
\end{equation}
for any $D>0$ if $N\ge N_1(D)$.

Then, for any $\epsilon>0$, the estimates
\begin{align}
|r_A^c(\widehat{z})| + |r_B^c(\widehat{z})|&\le \frac{N^\epsilon}{N^2\widehat{\eta}^3}\,, \label{071501} \\
|\omega_A^c(\widehat{z})-\omega_A(\widehat{z})|+|\omega_B^c(\widehat{z})-\omega_B(\widehat{z})|&\le\frac{N^\epsilon}{N^2\widehat{\eta}^3}\,, \label{071518} \\
 |\mathbb{E}m_H(\widehat{z})-m_{A\boxplus B}(\widehat{z})| &\le\frac{N^\epsilon}{N^2\widehat{\eta}^3}\,, 
 \label{Emvar}
 \end{align}
 hold for any $N\ge N_2(\epsilon)$. Moreover, for any $\epsilon, D>0$, the event
 \begin{equation}\label{071519}
 \Theta(\wh\eta)\equiv \Theta_E(\wh\eta)\deq \Xi_E(\wh\eta) \cap \Big\{ |m_H(\widehat{z})-m_{A\boxplus B}(\widehat{z})|\ge \frac{N^\epsilon}{\sqrt{N^2\widehat{\eta}^3}} \Big\}
\end{equation}
 satisfies
\begin{equation}\label{Om}
\P\big( \Theta(\wh\eta)\big) \le N^{-D}\,,
\end{equation}
 if $N\ge N_3(\epsilon, D)$. The threshold functions $N_1, N_2, N_3$ depend only on $\mu_\alpha, \mu_\beta$, the
speed of convergence in \eqref{le assumptions convergence empirical measures} and they 
 are uniform in $\wh\eta\in  [\eta_\mathrm{m},\eta_{\mathrm{M}}]$ and $E\in\mathcal{I}$.
\end{lemma}

We postpone the proof of Lemma \ref{lem071701} and prove Theorem \ref{thm041801} first.

\begin{proof}[Proof of Theorem \ref{thm041801}]  
We start with observing that it is sufficient to prove a version of~\eqref{le local convergence of m}
where the real part of the spectral parameter 
$E$  is fixed.
This version asserts that there is a 
 large ($N$-independent)~$\eta_{\mathrm{M}}$, to be fixed below, such that for any (small) $\epsilon>0$ and (large)~$D$,
 and  any fixed $E\in \mathcal{I}$,
\begin{align}\label{le local convergence of m new}
\mathbb{P}\,\bigg(\bigcup_{z\in\mathcal{S}_{E}(\eta_\mathrm{m},\eta_{\mathrm{M}})}\bigg\{\big|m_H(z)-m_{\mu_A\boxplus \mu_B}(z)\big|> \frac{N^{\epsilon}}{N (\im z)^{3/2}}\bigg\}\bigg)\le \frac{1}{N^D}\,,
\end{align}
holds for $N\ge N_0$, \ie the set $\mathcal{S}_{\mathcal{I}}(\eta_\mathrm{m},\eta_{\mathrm{M}})$ 
in~\eqref{le local convergence of m}
is replaced with $\mathcal{S}_{E}(\eta_\mathrm{m},\eta_{\mathrm{M}})\deq\{ E+\ii \eta \; : \; \eta\in[\eta_{\mathrm{m}},\eta_{\mathrm{M}} ]\}$. The threshold $N_0$ depends on $\epsilon, D$, $\mu_{\alpha}, \mu_\beta$, $\mathcal{I}$
and on the speed of convergence in \eqref{le assumptions convergence empirical measures}. 

Indeed, by introducing the discretized lattice version
$$
    \widehat{\mathcal{S}}_{\mathcal{I}}(a,b): =  {\mathcal{S}}_{\mathcal{I}}(a,b)\cap N^{-5}\{\Z\times \ii \Z\}
$$
of the spectral domain ${\mathcal{S}}_{\mathcal{I}}(a,b)$ (\cf \eqref{le domain S}) and by taking a union
bound, we see that~\eqref{le local convergence of m new}
implies  
\begin{align}\label{le local convergence of m newe}
\mathbb{P}\,\bigg(\bigcup_{z\in\widehat{\mathcal{S}}_{\mathcal{I}}(\eta_\mathrm{m},\eta_{\mathrm{M}})}\bigg\{\big|m_H(z)-m_{\mu_A\boxplus \mu_B}(z)\big|> \frac{N^{\epsilon}}{N (\im z)^{3/2}}\bigg\}\bigg)\le \frac{C}{N^{D-5}}\,.
\end{align}
Thanks to the Lipschitz continuity of the Stieltjes transforms $m_H(z)$ and $m_{\mu_A\boxplus \mu_B}(z)$
  with Lipschitz constant $\eta^{-2}=(\im z)^{-2}\le N^2$, for any $\im z\ge \eta_{\mathrm{m}}$, we see that~\eqref{le local convergence of m} follows from~\eqref{le local convergence of m newe} 
after a small adjustment of $\epsilon$ and $D$ that were anyway arbitrary.

 From now on we fix  $E\in \mathcal{I}$ and our goal is to prove \eqref{le local convergence of m new}. We will use   Lemma~\ref{lem071701}. In the first step we verify 
 that the assumptions  of this lemma hold for $\widehat\eta =\eta_{\mathrm{M}}$, \ie
 that~\eqref{071502} and~\eqref{Xi}  hold for $z=E+\ii \eta_{\mathrm{M}}$. In the second step,
 we successively use Lemma~\ref{lem071701} to reduce $\widehat\eta$ steps by steps of size $N^{-5}$
 until we have verified  \eqref{071502}--\eqref{Xi} down to $\widehat\eta =\eta_{\mathrm{m}}$.
 Then \eqref{le local convergence of m new} will follow from 
 a final application of Lemma~\ref{lem071701} combined with discretization argument 
  similar to the one above, but this time the $\eta$ variable instead of the $E$ variable.

 {\it Step 1. Initial bound.}
   First we note that 
  since $\mu_A$ and $\mu_B$ are compactly supported, $\|H\|$ is deterministically bounded, we thus have $\Im m_H(E+\ii \eta_{\mathrm{M}})\le(\eta_{\mathrm{M}})^{-1}\le 1$ assuming $\eta_{\mathrm{M}}\ge 1$.

 Following the main argument in the proof of Proposition~\ref{prop041901}, we have the concentration inequality 
\begin{align}
|f_{VQV^*}(E+\ii\eta_{\mathrm{M}})-\mathbb{E} f_{VQV^*}(E+\ii\eta_{\mathrm{M}})|\prec \frac{1}{\sqrt{N^2\eta_{\mathrm{M}}^{3}}}\,, \label{072202}
\end{align}
uniformly for any deterministic $Q$ with $\|Q\|\lesssim 1$. The analog concentration holds with~$V$ replaced by~$U$.  Using~\eqref{072202} with $Q=I$ ($I$ the identity matrix), we have $|\IE[m_H(E+\ii \eta_{\mathrm{M}})]|\prec{N^{-1}} $. Hence, it suffices to show that
\begin{align}
|\mathbb{E}m_H(E+\ii \eta_{\mathrm{M}})-m_{A\boxplus B}(E+\ii \eta_{\mathrm{M}})|\prec \frac{1}{N}\,. \label{072205}
\end{align}

Recalling the definitions of $\omega_A^c$ and $\omega_B^c$ in~\eqref{060104}, we have, with $z=E+\ii\eta_{\mathrm{M}}$, the expansion
\begin{align*}
 \omega_A^c(z)&=z-\frac{\E\,\ntr\! \widetilde A G_{\widetilde H}(z) }{\E\, \ntr G_{\widetilde H}(z)}=z-\frac{\ntr A z^{-1}+\E\,\ntr \widetilde A (\widetilde A+\widetilde B)z^{-2}}{z^{-1}}+O(z^{-2})\,,
\end{align*}
as $\eta_{\mathrm{M}}\nearrow\infty$. Thus using the assumption $\ntr A=0$ we get
\begin{align*}
 \im \omega_A^c(E+\ii\eta_{\mathrm{M}})-\im \eta_{\mathrm{M}}=\frac{\ntr A^2\eta_{\mathrm{M}}+\E\,\ntr\widetilde A\widetilde B\eta_{\mathrm{M}}}{|E+\ii\eta_{\mathrm{M}}|^2}+O\left(\frac{1}{\eta_{\mathrm{M}}^2}\right)\,,
\end{align*}
as $\eta_{\mathrm{M}}\nearrow\infty$. Next, since $V$ and $U$ are independent, we have $$\E\,\ntr VAV^* UBU^*=\ntr \E [VAV^*]\, \E[ UBU^*]=\ntr A\,\ntr B=0\,,$$
since $\ntr A=\ntr B=0$ by assumption. Thus
\begin{align}
 \im \omega_A^c(E+\ii\eta_{\mathrm{M}})-\im \eta_{\mathrm{M}}=\frac{\ntr A^2\eta_{\mathrm{M}}}{|E+\ii\eta_{\mathrm{M}}|^2}+O\left(\frac{1}{\eta_{\mathrm{M}}^2}\right)\,,
\end{align}
as $\eta_{\mathrm{M}}\nearrow\infty$. Since $\ntr A^2>0$, we achieve by choosing $\eta_{\mathrm{M}}$ sufficiently large (but independent of $N$) that
\begin{align}\label{corviglia}
 \im \omega_A^c(E+\ii\eta_{\mathrm{M}})-\im \eta_{\mathrm{M}}\ge\frac{1}{2}\frac{\ntr A^2\eta_{\mathrm{M}}}{|E+\ii\eta_{\mathrm{M}}|^2}\,,
\end{align}
and the analogue estimate holds with $A$ replaced by $B$. In particular, we have, for such $\eta_{\mathrm{M}}$,
\begin{align}\label{072206}
\im \omega_A^c(E+\ii\eta_{\mathrm{M}})\gtrsim 1\,,\qquad \im \omega_B^c(E+\ii\eta_{\mathrm{M}})\gtrsim 1\,,
\end{align}
and $\omega_A^c(E+\ii\eta_{\mathrm{M}})\sim 1$, $\omega_B^c(E+\ii\eta_{\mathrm{M}})\sim 1$.

To show~\eqref{072205}, we apply Lemma~\ref{le lemma large eta} to the system~\eqref{071517}. Having established~\eqref{corviglia}, it suffices to show that
\begin{align}
|r_A^c(E+\ii \eta_{\mathrm{M}})|\prec \frac{1}{N} \,,\qquad\qquad |r_B^c(E+\ii \eta_{\mathrm{M}})|\prec \frac{1}{N}\,,  \label{072001}
\end{align}
since then we have, for $N$ sufficiently large and $\eta_{\mathrm{M}}$ as above that, for any fixed $\varepsilon\in[0,1)$,
\begin{align}\label{072001k}
\frac{N^\varepsilon}{N}\le\frac{1}{2}\frac{\ntr A^2\eta_{\mathrm{M}}}{|E+\ii\eta_{\mathrm{M}}|^2}\le \im \omega_A^c(E+\ii\eta_{\mathrm{M}})-\im \eta_{\mathrm{M}}\,,
\end{align}
and similarly with $B$ replacing $A$. In particular, combining~\eqref{072001} and~\eqref{072001k}, we see that assumption~\eqref{le lower bound on tilde omega} of Lemma~\ref{le lemma large eta} (with the choice $\widetilde r = r^c$) is satisfied for $N$ sufficiently large (with high probability). Consequently, we see that~\eqref{071502} 
(even with $N^{-1+\epsilon}$ instead of $N^{-\gamma}$ in the latter) hold for 
 $z=E+\ii \eta_{\mathrm{M}}$. Finally, the equations
\begin{align}\label{le ww}
\mathbb{E} m_H(z)=\frac{1}{z-\omega_A^c(z)-\omega_B^c(z)}\,,\qquad \quad m_{A\boxplus B}(z)=\frac{1}{z-\omega_A(z)-\omega_B(z)}\,,
\end{align}
together with the concentration estimate~\eqref{072202} yield \eqref{Xi}.

It remains to justify~\eqref{072001}. Since $\ii\eta_{M}\,m_H(E+\ii\eta_{\mathrm{M}})=O(\eta_{\mathrm{M}}^{-1})$, we have $\mathbb{E}m_H(E+\ii \eta_{\mathrm{M}})\sim 1$. In addition, from~\eqref{072206} it follows that $m_A(\omega_B^c(E+\ii\eta_{\mathrm{M}}))\sim 1$ and $m_B(\omega_A^c(E+\ii\eta_{\mathrm{M}}))\sim 1$. Recalling the definition in~\eqref{071508}, we see that~\eqref{072001} is equivalent to 
\begin{align}
|\ntr \delta_A^c(E+\ii \eta_{\mathrm{M}})|\prec \frac{1}{N} \,,\qquad \qquad |\ntr \delta_B^c(E+\ii \eta_{\mathrm{M}})|\prec \frac{1}{N}\,.  \label{072002}
\end{align}
By the definitions of $\delta_A^c$, $\delta_B^c$ in~\eqref{041301}, and $\Delta_A$, $\Delta_B$ in~\eqref{041905}, it is easy to obtain~\eqref{072002} by using~\eqref{072202} and Cauchy--Schwarz.
This completes {\it Step 1}, \ie the verification of \eqref{071502}--\eqref{Xi} for $\wh\eta=\eta_{\mathrm{M}}$.

 {\it Step 2. Induction.} Recall that $\omega_A$, $\omega_B$ and $m_{A\boxplus B}$ (see Lemma~\ref{cor.080601})
are uniformly bounded and $\omega_A^c(z)$, $\omega_B^c(z)$, $m_H(z)$, $\omega_A(z)$, $\omega_B(z)$ and $m_{A\boxplus B}(z)$
are Lipschitz continuous with a Lipschitz constant bounded by $(\im z)^{-2}\le N^2$, for any $\im z\ge \eta_{\mathrm{m}}$.
 Applying Lemma~\ref{lem071701}  to conclude~\eqref{071518} with  the choice $\epsilon=\gamma/10$, we see that if~\eqref{071502} and~\eqref{Xi} hold for some $\wh\eta$,
  then ~\eqref{071502}  also holds
   for $\wh\eta$ replaced with $\widehat{\eta}-N^{-5}$ as long as $\widehat{\eta}\geq \eta_\mathrm{m}$.
  Moreover, by the Lipschitz continuity of $m_H$ and $m_{A\boxplus B}$, notice that
 \begin{equation}\label{contain}
     \Xi(\wh\eta - N^{-5}) \supset \Xi(\wh\eta) \setminus\Theta(\wh\eta)\,.
 \end{equation}
Thus, if \eqref{Xi} holds for some $\wh\eta$, then~\eqref{contain} and~\eqref{Om} imply that~\eqref{Xi} also holds for~$\wh\eta$ replaced with $\widehat{\eta}-N^{-5}$. 
Using {\it Step~1} as an initial input with the choice~$\widehat{\eta} = \eta_{\mathrm{M}}$, 
and applying the above induction argument $O(N^5)$ times by reducing~$\widehat\eta$ with stepsize~$N^{-5}$,
we  see  that~\eqref{071502} and~\eqref{Xi}  hold for all $\wh\eta_k\in [\eta_\mathrm{m},\eta_{\mathrm{M}}]$
of the form $\wh\eta_k = \eta_{\mathrm{M}} - k\cdot N^{-5}$ with some integer $k$.
Applying Lemma~\ref{lem071701} once more
for these $\wh\eta_k$, but now with an arbitrary $\epsilon>0$, we conclude from \eqref{Xi} and  \eqref{Om} that
\begin{equation}\label{mm}
  |m_H(E+\ii \wh\eta_k)-m_{A\boxplus B}(E+\ii \wh\eta_k)|\prec\frac{1}{\sqrt{N^2\wh\eta_k^3}}\,, \qquad\quad k=0,1,\ldots, k_0\,,
\end{equation}
where $k_0$ is the largest integer with $\wh\eta_{k_0}\ge \eta_{\mathrm{m}}$.
The uniformity of \eqref{mm} in  $k$ follows from the fact that  the threshold functions $N_j$ in 
Lemma~\ref{lem071701} are independent of $\wh\eta$.
 Clearly $k_0=O(N^5)$, so taking a union bound of \eqref{mm}, compensating
the combinatorial  factor $CN^{5}$ by replacing $D$ by $D-5$, and slightly adjusting $\epsilon$ to
extend the control from the set $\{ z=E+\ii \wh\eta_k \; : \; k\le k_0\}$ to all $z\in\mathcal{S}_E( \eta_{\mathrm{m}},\eta_{\mathrm{M}})$,
we obtain \eqref{le local convergence of m new}.
\qedhere \end{proof}

It remains to prove Lemma~\ref{lem071701}. 

\begin{proof}[Proof of Lemma \ref{lem071701}] 
First we notice that $E\in\mathcal{I}$ and~\eqref{le assumptions convergence empirical measures} imply
that for all sufficiently large $N$, the bounds~\eqref{le upper bound on omega AB}-\eqref{le lower bound on omega AB}
hold. Together with~\eqref{071502} they imply that 
   \begin{align}\label{071503}
&\omega_A^c(\wh z)\,,\omega_B^c(\wh z)\sim 1\,, \qquad\Im \omega_A^c(\wh z)\,, \Im\omega_B^c(\wh z)\gtrsim 1\,,  
\end{align}
moreover, using~\eqref{le third equation} we also get
\begin{equation}\label{Em}
   \frac{1}{\E m_H(\wh z)}\lesssim 1.
\end{equation}

We start with~\eqref{071501}. Thanks to symmetry, we only need to estimate~$|r_A^c(\wh z)|$. By~\eqref{071503} we have 
\begin{align}
\|G_{\widetilde{A}}(\omega_B^c(\widehat{z}))\|=\|G_{A}(\omega_B^c(\widehat{z}))\|\lesssim 1\,. \label{071511}
\end{align}
Furthermore, $\omega_B^c(\widehat{z})\sim 1$ and $\Im \omega_B^c(\widehat{z})\gtrsim 1$ imply $m_A(\omega_B^c(\widehat{z}))\sim 1$.

We next claim that
\begin{align}
\Big|\mathbb{E}\big[\ntr\!\big( G_{\widetilde{A}}(\omega_B^c(\widehat{z}))(\widetilde{A}-\widehat{z})\Delta_A(\widehat{z})\big)\big]\Big|\leq \frac{N^\epsilon}{N^2\widehat{\eta}^3}\,, \label{071507}
\end{align}
for any $\epsilon>0$ if $N\ge N_0(\epsilon)$ is large enough, uniformly for
 $\wh\eta\in[\eta_{\mathrm{m}},\eta_{\mathrm{M}}]$.
Assuming~\eqref{071507} and recalling the definition of $\delta_A^c$ and $r_A^c$ in~\eqref{041301}-\eqref{071508}, 
from \eqref{Em} we get the first estimate in~\eqref{071501}. 

Next, we prove~\eqref{071507}. By the definitions in~\eqref{041905}, we have
\begin{align}
\mathbb{E}\Big[\ntr\!\big(G_{\widetilde{A}}(\omega_B^c(\widehat{z}))(\widetilde{A}-\widehat{z})\Delta_A(\widehat{z})\big)\Big]=&-\mathbb{E}\Big[\IE[m_H(\widehat{z})]\,\ntr\!\big( G_{\widetilde{A}}(\omega_B^c(\widehat{z}))(\widetilde{A}-\widehat{z})G_{\widetilde H}(\widehat{z})\big)\Big]\nonumber\\
&-\mathbb{E} \Big[\IE[f_B(\widehat{z})]\,\ntr\!\big(G_{\widetilde{A}}(\omega_B^c(\widehat{z}))G_{\widetilde H}(\widehat{z})\big) \Big]\,.\label{041920}
\end{align}
We rewrite the two terms on the right side separately as covariances,
\begin{multline*}
\mathbb{E}\Big[\IE[m_H(\widehat{z})]\,\ntr\!\big( G_{\widetilde{A}}(\omega_B^c(\widehat{z}))(\widetilde{A}-\widehat{z})G_{\widetilde H}(\widehat{z})\big)\Big]
=\mathrm{Cov}\left(m_H(\widehat{z}),\ntr\!\left(G_{\widetilde{A}}(\omega_B^c(\widehat{z}))(\widetilde{A}-\widehat{z})G_{\widetilde H}(\widehat{z})\right)\right)\,,\end{multline*}
respectively,
\begin{align*}
 \mathbb{E}\Big[\IE[f_{\widetilde B}(\widehat{z})] \ntr\!\left(G_{\widetilde{A}}(\omega_B^c(\widehat{z}))G_{\widetilde H}(\widehat{z})\right)\Big]
=\mathrm{Cov}\Big(f_{\widetilde B}(\widehat{z}), \ntr\!\big(G_{\widetilde{A}}(\omega_B^c(\widehat{z}))G_{\widetilde H}(\widehat{z})\big)\Big)\,,
\end{align*}
where $\mathrm{Cov}(X,Y)\deq\mathbb{E}(\IE[X]\cdot\IE[Y])$, for arbitrary random variables $X$ and $Y$.

Given~\eqref{Xi} and the uniform boundedness of $m_{A\boxplus B}(z)$ from \eqref{le bound on mAB stuff},
 we see that~\eqref{041816} is satisfied and   we can apply Proposition~\ref{prop041901}  
using different choices for~$Q$.
Together with Cauchy--Schwarz inequality $|\mathrm{Cov}(X,Y)|^2\leq \mathbb{E}|\!\IE[X]|^2\cdot\mathbb{E}|\!\IE[Y]|^2$, we get 
\begin{align}\label{071515}
\Big|\mathrm{Cov}\left(m_H(\widehat{z}),\ntr\! \left(G_{\widetilde{A}}(\omega_B^c(\widehat{z}))(\widetilde{A}-\widehat{z})G_{\widetilde H}(\widehat{z})\right)\right)\Big|&\prec\frac{1}{N^2\widehat{\eta}^3}\,,\nonumber\\
 \Big|\mathrm{Cov}\left(f_{\widetilde B}(\widehat{z}), \ntr\!\left(G_{\widetilde{A}}(\omega_B^c(\widehat{z}))G_{\widetilde H}(\widehat{z})\right)\right)\Big|&\prec\frac{1}{N^2\widehat{\eta}^3}\,. 
\end{align} 
More specifically, for the first line of \eqref{071515}, 
we chose $Q=I$ and $Q=G_{A}(\omega_B^c(\widehat{z}))(A-\widehat{z})$; 
for the second line we chose $Q=B$ and $Q =G_{A}(\omega_B^c(\widehat{z}))$, where we also used the facts $\widetilde{A}=VAV^*$ and $\widetilde{B}=UBU^*$. Here, we also implicitly used~\eqref{071511}. Then, 
\eqref{071507} follows from~\eqref{071515}, which in turn proves~\eqref{071501}. 

Next, using Proposition~\ref{le proposition perturbation of system},~\eqref{071517} and ~\eqref{071501} we immediately get~\eqref{071518}. Moreover, since $\im \omega_A(z)\,,\im \omega_B(z)\ge \im z$, we have $|z-\omega_A(z)-\omega_B(z)|\geq \Im \omega_B(z)\gtrsim 1$. Together with~\eqref{071518} and~\eqref{le ww} 
this implies~\eqref{Emvar}. Notice that~\eqref{Xi} together with the uniform bound on $m_{A\boxplus B}$
imply the condition \eqref{041816} in Proposition~\ref{prop041901}.
Thus,  finally~\eqref{071519} and~\eqref{Om} follow from~\eqref{Emvar}
 and the concentration inequality~\eqref{041815}. This completes the proof of Lemma~\ref{lem071701}.\qedhere\end{proof}

\section{Two point mass case}\label{le final section}
In this section, we discuss stability properties of the free additive convolution $\mu_\alpha\boxplus\mu_\beta$ when both $\mu_\alpha$ and $\mu_\beta$ are convex combinations of two point masses. The analogous result to Theorem~\ref{thm stability} is given in Proposition~\ref{stability for two point mass} below. Applications of that result in the spirit of Theorems~\ref{le theorem continuity} and~\ref{thm041801} are then stated in Proposition~\ref{le theorem continuity: two point mass case} and Proposition~\ref{pro.a.1}. When we refer to the results in Sections \ref{section: Main results}-\ref{section perturbation}, we will henceforth regard $\mu_1$ and $\mu_2$ as $\mu_\alpha$ and $\mu_\beta$, respectively, unless specified otherwise.

\subsection{Stability in the two point mass case}
Without loss of generality (up to shifting and scaling), we assume that
\begin{align}
&\mu_\alpha=\xi\delta_1+(1-\xi)\delta_0\,,\qquad \mu_\beta=\zeta\delta_{\theta}+(1-\zeta)\delta_0\,,\qquad \theta\neq 0\,,\qquad\nonumber\\
 & \xi,\zeta\in (0,\frac{1}{2}]\,, \qquad \xi\leq \zeta,\qquad (\theta, \xi,\zeta)\neq (-1,\frac12,\frac12)\,. \label{081210}
\end{align}
Here we exclude the case  $(\theta, \xi,\zeta)=(-1,\frac12,\frac12)$ since it is equivalent to $(\theta, \xi,\zeta)=(1,\frac12,\frac12)$ under a shift. Note that the latter is a special case of $\mu_\alpha=\mu_\beta$. 

Set
\begin{align*}
\ell_{1}&\deq\min\Big\{\frac{1}{2}\Big(1+\theta-\sqrt{(1-\theta)^2+4\theta r_+}\Big), \frac{1}{2}\Big(1+\theta-\sqrt{(1-\theta)^2+4\theta r_-}\Big)\Big\}\,,\\
\ell_{2}&\deq\max\Big\{\frac{1}{2}\Big(1+\theta-\sqrt{(1-\theta)^2+4\theta r_+}\Big), \frac{1}{2}\Big(1+\theta-\sqrt{(1-\theta)^2+4\theta r_-}\Big)\Big\}\,,\\
\ell_{3}&\deq\min\Big\{\frac{1}{2}\Big(1+\theta+\sqrt{(1-\theta)^2+4\theta r_+}\Big), \frac{1}{2}\Big(1+\theta+\sqrt{(1-\theta)^2+4\theta r_-}\Big)\Big\}\,,\\
\ell_{4}&\deq\max\Big\{\frac{1}{2}\Big(1+\theta+\sqrt{(1-\theta)^2+4\theta r_+}\Big), \frac{1}{2}\Big(1+\theta+\sqrt{(1-\theta)^2+4\theta r_-}\Big)\Big\}\,,
\end{align*}
where we introduced
\begin{align}
r_{\pm}\deq\xi+\zeta-2\xi\zeta\pm\sqrt{4\xi\zeta(1-\xi)(1-\zeta)}\,. \label{r plus minus}
\end{align}
Note that $\ell_1<\ell_2\le \ell_3<\ell_4$. The following result, taken from~\cite{Kargin2012a}, describes the regular bulk of $\mu_\alpha\boxplus\mu_\beta$ in the  setting of~\eqref{081210}. Recall that $f_{\mu_\alpha\boxplus\mu_\beta}$ denotes the density of $(\mu_\alpha\boxplus\mu_\beta)^{\mathrm{ac}}$.

\begin{lemma} \label{support of two point mass} Let $\mu_\alpha$ and $\mu_\beta$ be as in~\eqref{081210}. Then the regular bulk is given by
\begin{align}
\mathcal{B}_{\mu_\alpha\boxplus\mu_\beta}=(\ell_{1}, \ell_{2})\cup (\ell_{3}, \ell_{4})\,, \label{def of I case 1}
\end{align}
in case  $\mu_\alpha\neq \mu_\beta$, while in case $\mu_\alpha=\mu_\beta$ it is given by
\begin{align}
\mathcal{B}_{\mu_\alpha\boxplus\mu_\alpha}=(\ell_{1}, \ell_{4}) \label{def of I case 2}\,.
\end{align}

 \end{lemma}

\begin{proof} Choose the diagonal matrices $A$ and $B$ with spectral distribution $\mu_A=\xi_N\delta_1+(1-\xi_N)\delta_0$ and $\mu_B=\zeta_N\delta_\theta+(1-\zeta_N)\delta_0$ respectively, with $\xi_N\deq\lfloor \xi N\rfloor/N$ and $\zeta_N\deq\lfloor \zeta N\rfloor/N$, where $\lfloor\,\cdot\,\rfloor$ denotes the integer part. Recall from~\eqref{081210} that $\xi\leq \zeta$ and  $\xi+\zeta\leq 1$.
From Theorem~1.1 of~\cite{Kargin2012a}, we first observe that the $\theta$ and $0$ are eigenvalues of the matrix $H=A+UBU^*$, $U$ a Haar unitary, with multiplicities $N(\zeta_N-\xi_N)$ and $N(1-\zeta_N-\xi_N)$, respectively. The remaining $2\xi_N N$ eigenvalues of $H$ may be obtained via a two-fold transformation from the eigenvalues, $(t_i)$, of a $\xi_N N$-dimensional Jacobi ensemble as
\begin{align}
\tau_{ j}^{\pm}\deq\frac{1}{2}\Big(1+\theta\pm \sqrt{(1-\theta)^2+4 t_j}\Big)\,,\qquad j=1,\ldots, \xi_N N\,, \label{two fold transformation}
\end{align}
and then identifying the eigenvalues of $H$ as the set $\{\tau_j^{+}\}\cup \{\tau_j^{-}\}\cup\{0,\theta\}$.
In addition, the weak limit of $\frac{1}{\xi_N N}\sum_j \delta_{t_j}$, as $N\to\infty$, admits a density given by 
\begin{align}
f(x)=\frac{1}{2\pi\xi}\frac{\sqrt{(r_+-x)(x-r_-)}}{x(1-x)} \mathbf{1}_{[r_-,r_+]}(x)\,,\qquad\qquad x\in\R\,, \label{density of Jacobi}
\end{align}
where $r_+$ and $r_-$ are defined in~\eqref{r plus minus}. Since the limiting spectral distribution of $H$ is given by $\mu_\alpha\boxplus\mu_\beta$, we see that  $(\mu_\alpha\boxplus\mu_\beta)^{\mathrm{ac}}$ agrees with the weak limit of the measure $\frac{1}{N}\sum_{j}(\delta_{\tau_j^+}+\delta_{\tau_{j}^-})$, as $N\to\infty$. 
Using this information together with~\eqref{two fold transformation} and~\eqref{density of Jacobi}, one deduces that $\mathrm{supp}\, (\mu_\alpha\boxplus\mu_\beta)^{\mathrm{ac}}=[\ell_1,\ell_2]\cup [\ell_3,\ell_4]$.  It then follows from the explicit form of the limiting distribution of the Jacobi ensemble that $f_{\mu_\alpha\boxplus\mu_\beta}$ is bounded and strictly positive inside its support. This proves~\eqref{def of I case 1}.

In the special case $\mu_\alpha=\mu_\beta$, we have $\ell_2=\ell_3=1$ and thus $\mathrm{supp}\,(\mu_\alpha\boxplus\mu_\beta)^{\mathrm{ac}}=[\ell_1,\ell_4]$, with $\ell_1=1-2\sqrt{\xi(1-\xi)}$ and $\ell_4=1+2\sqrt{\xi(1-\xi)}$. In fact, the density of $(\mu_\alpha\boxplus\mu_\alpha)^{\mathrm{ac}}$ equals
\begin{align}\label{le fa density}
f_{\mu_\alpha\boxplus\mu_\alpha}(x)=\frac{1}{\pi}\frac{\sqrt{(\ell_4-x)(x-\ell_1)}}{x(2-x)}\,,\qquad\qquad x\in (\ell_1,\ell_4)\,;
\end{align}
see~(5.5) of~\cite{VP} for instance. Then~\eqref{def of I case 2} follows directly
\end{proof}

Our main task in this section is to show the following  result on the stability of the system $\PP_{\mu_\alpha,\mu_\beta}(\omega_\alpha,\omega_\beta,z)=0$ in the setting~\eqref{081210}. Recall the definition of $\Gamma_{\mu_\alpha,\mu_\beta}$ in~\eqref{le what stable means}.
\begin{proposition} \label{stability for two point mass}
Let $\mu_\alpha$ and $\mu_\beta$ be as in~\eqref{081210}. Let $\mathcal{I}\subset\mathcal{B}_{\mu_\alpha\boxplus\mu_\beta}$ be a compact non-empty interval. Fix $0< \eta_{\mathrm{M}}<\infty$. Then, there are constants $k>0$, $K<\infty$ and $S<\infty$, depending on the constants $\xi$, $\zeta$, $\theta$, $\eta_M$ and on the interval $\mathcal{I}$, such that  the subordination functions possess the following bounds:
\begin{align}\label{two point mass different: minimum of omegas}
 \min_{z\in \mathcal{S}_{\mathcal{I}}(0,\eta_{\mathrm{M}})}\im \omega_\alpha(z)\ge 2k\,,\qquad \min_{z\in \mathcal{S}_{\mathcal{I}}(0,\eta_{\mathrm{M}})}\im \omega_\alpha(z)\ge 2k\,,
\end{align}
\begin{align}\label{two point mass different: maximum of omegas}
 \max_{z\in \mathcal{S}_{\mathcal{I}}(0,\eta_{\mathrm{M}})}|\omega_\alpha(z)|\le \frac{K}{2}\,,\qquad\max_{z\in \mathcal{S}_{\mathcal{I}}(0,\eta_{\mathrm{M}})}|\omega_\beta(z)|\le\frac{K}{2}\,.
\end{align}
Moreover, we have the following bounds:
\begin{itemize}
\item[$(i)$] If $\mu_\alpha\neq \mu_\beta$,
\begin{align} \label{bound of gamma for special case 1}
\max_{z\in \mathcal{S}_{\mathcal{I}}(0,\eta_{\mathrm{M}})}\Gamma_{\mu_\alpha,\mu_\beta}(\omega_\alpha(z),\omega_\beta(z))\le S\,.
\end{align}
\item[$(ii)$] If $\mu_\alpha= \mu_\beta$, 
\begin{align} \label{bound of gamma for special case 2}
\Gamma_{\mu_\alpha,\mu_\alpha}(\omega_\alpha(z),\omega_\alpha(z))\le \frac{S}{|z-1|}\,,
\end{align}
holds uniformly on $\mathcal{S}_{\mathcal{I}}(0,\eta_{\mathrm{M}})$.
\end{itemize}
\end{proposition}

\begin{remark}
As an immediate consequence of Proposition~\ref{stability for two point mass} and~\eqref{le baby system}, we obtain for $\mu_\alpha\not=\mu_\beta$ the bounds
$\max_{z\in\mathcal{S}_{\mathcal{I}}(0,\eta_{\mathrm{M}})}|\omega_\alpha'(z)|\le 2S$, $\max_{z\in\mathcal{S}_{\mathcal{I}}(0,\eta_{\mathrm{M}})}|\omega_\beta'(z)|\le 2S$
with $\mathcal{I}$ as in~\eqref{bound of gamma for special case 1}. For $\mu_\alpha=\mu_\beta$, we get $|\omega_\alpha'(z)|\le \frac{2S}{|z-1|}$,
uniformly on $\mathcal{S}_{\mathcal{I}}(0,\eta_{\mathrm{M}})$ as in~\eqref{bound of gamma for special case 2}.
\end{remark}
\begin{remark} \label{E=1 is unstable}
In the case $\mu_\alpha=\mu_\beta$, we note from Lemma~\ref{support of two point mass} (\cf~\eqref{le fa density}) that the point $E=1$ is in the regular bulk~$\mathcal{B}_{\mu_\alpha\boxplus\mu_\alpha}$. However, $m_{\mu_\alpha\boxplus\mu_\beta}(1+\ii0)$ is unstable under small perturbations. For instance, let
 \begin{align*}
 \mu_A=\mu_\alpha=\xi\delta_1+(1-\xi)\delta_0\,,\qquad\qquad \mu_B=(\xi-\varepsilon)\delta_1+(1-\xi+\varepsilon)\delta_0\,,
 \end{align*} 
 for some small $\varepsilon>0$. Then, according to Theorem 7.4 of \cite{BeV98}, $\mu_A\boxplus\mu_B$ has a point mass $\varepsilon\delta_1$. Hence, even though~\eqref{le assumptions convergence empirical measures} (\ie $\dL(\mu_B,\mu_\beta)\rightarrow 0$, as $\varepsilon\to 0$) is satisfied, $m_{\mu_A\boxplus\mu_B}(z)$ contains a singular part $\frac{\varepsilon}{(1-z)}$, which blows up as $|z-1|= o(\varepsilon)$. This explains, on a heuristic level, the bound in~\eqref{bound of gamma for special case 2} and shows why the $\mu_\alpha=\mu_\beta$ case at energy $E=1$ is special even though the density $f_{\mu_{\alpha}\boxplus\mu_\alpha}$ is real analytic in a neighborhood of $E=1$. 
\end{remark}

\begin{remark} Consider a more general setup with $\mu_\alpha=\xi\delta_a+\widetilde{\mu}_\alpha$ and $\mu_\beta=(1-\xi)\delta_b+\widetilde{\mu}_{\beta}$, for some constants $\xi\in (0,1)$, $a,b\in\mathbb{R}$ and for some Borel measures $\widetilde{\mu}_\alpha$ and $\widetilde{\mu}_\beta$ with 
$\widetilde{\mu}_\alpha(\mathbb{R})=1-\xi$ and $\widetilde{\mu}_\beta(\mathbb{R})=\xi$. Analogously to the discussion in Remark~\ref{E=1 is unstable}, we note that $m_{\mu_\alpha\boxplus\mu_\beta}(a+b+\ii0)$ is unstable under small perturbations. However, from Lemma \ref{lemma linear stability}, we know that the system $\PP_{\mu_\alpha,\mu_\beta}(\omega_\alpha,\omega_\beta,z)=0$ is linearly $S$-stable in the regular bulk under the assumptions of Theorem \ref{thm stability}. That means, if neither $\mu_\alpha$ nor $\mu_\beta$ is supported at a single point and at least one of them is supported at more than two points, then the point $E=a+b$ {\it cannot} lie in the regular bulk $\mathcal{B}_{\mu_\alpha\boxplus\mu_\beta}$. Thus, only in the special case $\mu_\alpha=\mu_\beta$ with $\mu_\alpha$ as in~\eqref{081210}, there is an unstable point, up to scaling and shifting given by $E=1$, inside the regular bulk $\mathcal{B}_{\mu_\alpha\boxplus\mu_\alpha}$.
\end{remark}

\begin{proof}[Proof of Proposition \ref{stability for two point mass}] Estimates~\eqref{two point mass different: minimum of omegas} and~\eqref{two point mass different: maximum of omegas} follow from Lemma~\ref{le lem upper bound on omega} and Lemma~\ref{le lemma stability}. 

To show statement $(i)$, we recall from the proof of Lemma~\ref{lemma linear stability} that $\PP_{\mu_\alpha,\mu_\beta}(\omega_\alpha,\omega_\beta,z)=0$ is linearly $S$-stable at $(\omega_\alpha,\omega_\beta)$ if
\begin{align}
\big|1-(F_{\mu_\alpha}'(\omega_\beta)-1) (F_{\mu_\beta}'(\omega_\alpha)-1) \big|\geq c\,, \label{081220}
\end{align}
for some strictly positive constant $c$. We now show that~\eqref{081220} holds for the case $\mu_\alpha\neq\mu_\beta$ in the setup of~\eqref{081210}. Using henceforth the shorthand $F_{\alpha}\equiv F_{\mu_\alpha}$, $F_{\beta}\equiv F_{\mu_\beta}$, we compute
\begin{align}
F_\alpha(z)=\frac{z(1-z)}{1-\xi-z}\,,\qquad\quad F_\beta(z)=\frac{z(\theta-z)}{\theta-\theta \zeta-z}\,,\qquad\qquad z\in\C^+\,. \label{080601}
\end{align}
Then it is easy to obtain
\begin{align}
F'_\alpha(z)-1=\frac{\xi-\xi^2}{(1-\xi-z)^2}\,,\qquad\quad F'_\beta(z)-1=\frac{\theta^2(\zeta-\zeta^2)}{(\theta-\theta \zeta-z)^2}\,, \label{072303}
\end{align}
and
\begin{align*}
|F'_\alpha(z)-1|=\frac{\Im F_\alpha(z)-\Im z}{\Im z}\,,\qquad\quad |F'_\beta(z)-1|=\frac{\Im F_\beta(z)-\Im z}{\Im z}\,.
\end{align*}
Consequently, we have (\cf~\eqref{le esel1})
\begin{align}
\Big|\big(F'_\alpha(\omega_\beta(z))-1\big)\big( F'_\beta(\omega_\alpha(z))-1\big)\Big|=\frac{(\Im\omega_\alpha(z)-\im z)(\Im \omega_\beta(z)-\im z)}{\Im\omega_\alpha(z)\Im \omega_\beta(z)}\, \label{072310}
\end{align}
for any $z\in\C^+$. Hence, for $z\in \mathcal{S}_{\mathcal{I}}(\eta_0, \eta_{\mathrm{M}})$ with some small but fixed $\eta_0>0$ to be chosen below, we trivially get~\eqref{081220} from~\eqref{072310}. It remains to discuss the regime $z\in \mathcal{S}_{\mathcal{I}}(0, \eta_0)$. Then~\eqref{080601} together with~\eqref{le definiting equations} implies that
\begin{align}
\frac{\omega_\beta(1-\omega_\beta)}{1-\xi-\omega_\beta}=\frac{\omega_\alpha(\theta-\omega_\alpha)}{\theta-\theta \zeta-\omega_\alpha}\,,\qquad \quad\frac{\omega_\beta(1-\omega_\beta)}{1-\xi-\omega_\beta}=\omega_\alpha+\omega_\beta-z\,. \label{07220111}
\end{align}
Denote $s\deq 1-\xi-\omega_\beta$ and $ t\deq\theta-\theta \zeta-\omega_\alpha$. 
From~\eqref{072303} we then have
\begin{align}
\big(F'_\alpha(\omega_\beta)-1\big)\big( F'_\beta(\omega_\alpha)-1\big)=\frac{(\xi-\xi^2)(\theta^2\zeta-(\theta \zeta)^2)}{(st)^2}\,. \label{072360}
\end{align}
Using~\eqref{07220111}, some algebra reveals that
\begin{align}
\frac{1}{st}=-\frac{1}{\xi-\xi^2}+\frac{\xi+\theta-\theta \zeta-z}{(\xi-\xi^2)t}\,, \qquad\qquad  \frac{1}{st}=-\frac{1}{\theta^2(\zeta-\zeta^2)}+\frac{\theta \zeta+1-\xi-z}{\theta^2(\zeta-\zeta^2)s}\,.\label{072206000}
\end{align}
Owing to~\eqref{072310} and $(\xi-\xi^2)(\theta^2\zeta-(\theta \zeta)^2)>0$ (recall that $\xi,\zeta\in(0,1/2]$ and $\theta\not=0$), it suffices to show that
\begin{align}
|\Im(st)|\geq c\,, \label{072205000}
\end{align}
in order to prove~\eqref{081220}. Note that, from the definitions of $s$ and $t$, together with~\eqref{two point mass different: minimum of omegas} and~\eqref{two point mass different: maximum of omegas}, we have
\begin{align}
|\Im s|\,, |\Im t|\geq c\,,\qquad\quad |s|\,, |t|\leq C\,. \label{080502}
\end{align}
Since $\mu_{\alpha}\neq \mu_{\beta}$, there exists a positive constant $d$ such that $\max\{|\xi-\zeta|,|\theta-1|\}\geq d$. It is then elementary to work out that
\begin{align}
\max \{|(\xi-\xi^2)-\theta^2(\zeta-\zeta^2)|\,,\,|(2\xi-2\theta \zeta+\theta-1)|\}\geq d_1\,, \label{081750}
\end{align}
for some positive constant $d_1\equiv d_1(\xi,\zeta,\theta)>0$, since the special case $(\theta,\xi,\zeta)=(-1,\frac12,\frac12)$ is also excluded in the setting~\eqref{081210}. For brevity, we adopt the notation
\begin{align*}
\varphi\deq\frac{\theta \zeta+1-\xi-z}{\theta^2(\zeta-\zeta^2)s}\,,\qquad \psi\deq\frac{\xi+\theta-\theta \zeta-z}{(\xi-\xi^2)t}\,.
\end{align*}
Then, according to~\eqref{072206000} we have
\begin{align}
\Re\frac{1}{st}=\Re \psi-\frac{1}{\xi-\xi^2}=\Re \varphi-\frac{1}{\theta^2(\zeta-\zeta^2)}\,,\qquad \Im \frac{1}{st}=\Im \psi=\Im \varphi\,. \label{081050}
\end{align}
If~$|(\xi-\xi^2)-\theta^2(\zeta-\zeta^2)|\geq d_1$ holds in~\eqref{081750}, then~\eqref{081050} implies that
\begin{align}
|\Re \psi-\Re \varphi|\geq d_2\,,\label{080605}
\end{align}
for some positive constant $d_2\equiv d_2(\xi,\zeta,\theta)$. For small enough $\eta_0=\eta_0(\xi,\zeta,\theta)$, we then get
\begin{align*}
\Re \psi-\Re \varphi=\frac{(\xi+\theta-\theta \zeta-E)\Re t+O(\eta_0)}{(\xi-\xi^2)|t|^2}-\frac{(\theta \zeta+1-\xi-E)\Re s+O(\eta_0)}{\theta^2(\zeta-\zeta^2)|s|^2}\,,
\end{align*}
which, together with~\eqref{080502} and~\eqref{080605}, implies that
\begin{align}
\max\big\{|\theta \zeta+1-\xi-E|\,,|\xi+\theta-\theta \zeta-E| \big\}\geq d_3\,, \label{080501}
\end{align}
for some positive constant $d_3\equiv d_3(\xi,\zeta,\theta)$. If, on the other hand, $|(2\xi-2\theta \zeta+\theta-1)|\geq d_1$ holds in~\eqref{081750}, we get~\eqref{080501} by triangle inequality. Either way,~\eqref{080501} follows from~\eqref{081750}, for sufficiently small, but fixed, $\eta_0>0$.

Next, using ~\eqref{080502} and~\eqref{080501}, we see that there is a constant $c>0$ such that, for sufficiently small $\eta_0$, for all $z\in \mathcal{S}_{\mathcal{I}}(0,\eta_0)$, we have $\max\{|\Im \varphi|,|\Im \psi|\}\geq c$. Since $\im \varphi=\im \psi$ by~\eqref{081050},~\eqref{072205000} holds on $\mathcal{S}_{\mathcal{I}}(0,\eta_0)$.  Therefore,~\eqref{072205000} holds on all of $\mathcal{S}_{\mathcal{I}}(0,\eta_{\mathrm{M}})$.
So, if $\mu_\alpha\not=\mu_\beta$, the system $\PP_{\mu_\alpha,\mu_\beta}(\omega_\alpha,\omega_\beta, z)=0$ is linearly $S$-stable with some finite $S$.

We next prove statement $(ii)$ where $\mu_\alpha=\mu_\beta$ and thus $\theta=1$, $\xi=\zeta$. From~\eqref{07220111}, we see that $\omega_\alpha=\omega_\beta$ satisfies the equation
\begin{align}
\frac{\omega_\alpha(1-\omega_\alpha)}{1-\xi-\omega_\alpha}=2\omega_\alpha-z\,. \label{equation for omega: special case}
\end{align}
Solving ~\eqref{equation for omega: special case} for $\omega_\alpha(z)$ we get
\begin{align}
\omega_\alpha(z)=\omega_\beta(z)=\frac{1}{2}\left({z-1+2(1-\xi)+\sqrt{(z-1)^2-4\xi(1-\xi)}}\right)\,, \label{solution of omega}
\end{align}
where the square root is chosen such that $\omega_\beta(z)\to \mathrm{i}\sqrt{\xi(1-\xi)}, $ as $z\to 1$.  Substituting~\eqref{solution of omega} into~\eqref{072360}, together with the $\theta=1$, $\zeta=\xi$, $s=t=1-\xi-\omega_\alpha$, yields
\begin{align*}
\big(F_{\alpha}'(\omega_\beta(z))-1\big) \big(F_{\beta}'(\omega_\alpha(z))-1\big)=\frac{4(\xi-\xi^2)^2}{\big(z-1+\sqrt{(z-1)^2-4(\xi-\xi^2)}\big)^4}\,.
\end{align*}
Then it is elementary to check that
\begin{align*}
\left|1-\big(F_{\alpha}'(\omega_\beta(z))-1\big) \big(F_{\beta}'(\omega_\alpha(z))-1\big) \right|\gtrsim |z-1|\,,\qquad\qquad z\in\mathcal{S}(0,\eta_{\mathrm{M}})\,,
\end{align*}
which further implies  $\Gamma_{\mu_\alpha,\mu_\beta}(\omega_\alpha(z),\omega_\beta(z))\lesssim 1/|z-1|$. Hence~\eqref{bound of gamma for special case 2} is proved.
\end{proof}

\subsection{Applications of Proposition~\ref{stability for two point mass}}
Analogously to Theorem \ref{thm stability}, we have two main applications of Proposition \ref{stability for two point mass}. The first one is the following modification of Theorem~\ref{le theorem continuity}. Let $\mu_\alpha$, $\mu_\beta$ be as in~\eqref{081210} and let $\mu_A$, $\mu_B$ be arbitrary probability measures on $\R$. Recall the domain $\mathcal{S}_{\mathcal{I}}(a,b)$ introduced in~\eqref{le domain S}. For given (small) $\varsigma>0$, we set
\begin{align}
\mathcal{S}_{\mathcal{I}}^{\varsigma}(a,b)\deq\bigg\{z\in \mathcal{S}_{\mathcal{I}}(a,b): \varsigma |z-1|\geq \max\Big\{ \sqrt{\dL(\mu_A,\mu_\alpha)}, \sqrt{\dL(\mu_B,\mu_\beta)} \Big\}\bigg\}\,.\label{081501}
 \end{align}

\begin{proposition}\label{le theorem continuity: two point mass case}
Let $\mu_\alpha,\mu_\beta$ be as in~\eqref{081210}. Let $\mathcal{I}\subset\mathcal{B}_{\mu_\alpha\boxplus\mu_\beta}$ be a compact non-empty interval.
Let $\mu_A$, $\mu_B$ be two probability measures on $\R$.  Fix $0<\eta_{\mathrm{M}}<\infty$. Then there are constants $b>0$ and $Z<\infty$ such that the condition
\begin{align}\label{new condition: two point mass}
\dL(\mu_A,\mu_\alpha)+\dL(\mu_B,\mu_\beta)\le b
\end{align}
implies
\begin{align}\label{le ksv statment: two point mass 1}
 \max_{z\in\mathcal{S}_{\mathcal{I}}(0,\eta_{\mathrm{M}})}\big|m_{\mu_A\boxplus \mu_B}(z)-m_{\mu_\alpha\boxplus\mu_\beta}(z)\big|\le Z\left(\dL(\mu_A,\mu_\alpha)+\dL(\mu_B,\mu_\beta)\right)\,,
\end{align}
in case $\mu_\alpha\neq \mu_\beta$, respectively
\begin{align}\label{le ksv statment: two point mass 2}
\big|m_{\mu_A\boxplus \mu_B}(z)-m_{\mu_\alpha\boxplus\mu_\alpha}(z)\big|\le \frac{Z}{|z-1|}\left(\dL(\mu_A,\mu_\alpha)+\dL(\mu_B,\mu_\alpha)\right)\,,
\end{align}
uniformly on $\mathcal{S}_{\mathcal{I}}^\varsigma(0,\eta_{\mathrm{M}})$ with $\varsigma\le \varsigma_0$, for some $\varsigma_0>0$, 
in case $\mu_\alpha= \mu_\beta$. The constants~$b$ and~$Z$ depend only on the constants $\xi,\zeta,\theta$ and on the interval $\mathcal{I}$, while~$\varsigma_0$ also depends on~$b$.
\end{proposition}
\begin{proof} Having established Proposition~\ref{stability for two point mass}, the proof of~\eqref{le ksv statment: two point mass 1} is the same as that of Theorem~\ref{le theorem continuity}. To establish~\eqref{le ksv statment: two point mass 2}, we mimic the proof of Theorem \ref{le theorem continuity} with $S$ replaced by $\frac{S}{|z-1|}$. We only give a sketch here. Similarly to~\eqref{zhigang today}, using~\eqref{two point mass different: minimum of omegas} and~\eqref{bound of gamma for special case 2}, we have with $b$ in~\eqref{new condition: two point mass} sufficiently small that
\begin{align}\label{081901}
\Gamma_{\mu_A,\mu_B}(\omega_\alpha(z),\omega_\alpha(z))\lesssim \frac{1}{|z-1|}\,,\qquad\qquad z\in\mathcal{S}_{\mathcal{I}}^\varsigma(0,\eta_{\mathrm{M}})\,.
\end{align}
As in the proof of Lemma~\ref{cor.080601}, we rewrite the system $\PP_{\mu_\alpha,\mu_\alpha}(\omega_\alpha(z),\omega_\alpha(z), z)=0$ as $\PP_{\mu_A,\mu_B}(\omega_\alpha(z),\omega_\alpha(z), z)=r(z)$ with $\|r(z)\|$
satisfying the bound~\eqref{080625}. From the uniqueness of the solution to $\PP_{\mu_A,\mu_B}(\omega_A,\omega_B, z)=0$ and~\eqref{081901}, we get
\begin{align}\label{difference between subordination functions}
|\omega_A(z)-\omega_\alpha(z)|&\lesssim \|r(z)\|/|z-1|\,,\nonumber\\
|\omega_B(z)-\omega_\alpha(z)|&\lesssim \|r(z)\|/|z-1|\,,\qquad\qquad z\in \mathcal{S}_{\mathcal{I}}^\varsigma(0,\eta_{\mathrm{M}}) \,,
\end{align}
via the Newton-Kantorovich theorem. Note that the inequality $\|r(z)\|\lesssim \varsigma^2 |z-1|^2$ is needed to guarantee that the first order term dominates over the higher order terms in the Taylor expansion of $\PP_{\mu_A,\mu_B}(\omega_A,\omega_B, z)$ around $\PP_{\mu_A,\mu_B}(\omega_\alpha,\omega_\beta, z)$. This is the reason why we restrict our discussion on the set $\mathcal{S}_{\mathcal{I}}^\varsigma(0,\eta_{\mathrm{M}})$. In addition, thanks to~\eqref{difference between subordination functions} we see that~\eqref{le upper bound on omega AB} and~\eqref{le lower bound on omega AB} still hold with $\mathcal{S}_{\mathcal{I}}(0,\eta_M)$ replaced by $\mathcal{S}_{\mathcal{I}}^\varsigma(0,\eta_{\mathrm{M}})$. Then the remaining parts of the proof of~\eqref{le ksv statment: two point mass 2} are  
the same as the counterparts in the proof of Theorem \ref{le theorem continuity}.
\end{proof}
The second application of Proposition \ref{stability for two point mass} gives the following local law for the Green function in the random matrix setup from Subsection~\ref{Application to random matrix theory}. Fix any $\gamma>0$. We introduce a sub-domain of $\mathcal{S}_{\mathcal{I}}^{\varsigma}(a,b)$ by setting
\begin{align}
 &\widetilde{\mathcal{S}}_{\mathcal{I}}^{\varsigma}(a,b)\deq \mathcal{S}_{\mathcal{I}}^{\varsigma}(a,b)\cap \Big\{z\in \mathbb{C}: |z-1|\geq \frac{N^{\gamma}}{\sqrt{N\eta^{3/2}}}\Big\}\,.
\end{align}

\begin{proposition} \label{pro.a.1}
Let $\mu_\alpha,\mu_\beta$ be as in~\eqref{081210}. Assume that the empirical eigenvalue distributions $\mu_A$, $\mu_B$ of the sequences of matrices $A$, $B$ satisfy~\eqref{le assumptions convergence empirical measures}. Fix any $0<\eta_{\mathrm{M}}<\infty$, any small $\gamma>0$ and set $\eta_\mathrm{m}=N^{-\frac23+\gamma}$. Let $\mathcal{I}\subset\mathcal{B}_{\mu_\alpha\boxplus\mu_\beta}$ be a compact non-empty interval. Then we have the following conclusions.
\begin{enumerate}
\item[$(i)$] If $\mu_{\alpha}\neq \mu_{\beta}$, then
\begin{align*}
&\max_{z\in\mathcal{S}_{\mathcal{I}}(\eta_\mathrm{m},\eta_{\mathrm{M}})}\big|m_H(z)-m_{A\boxplus B}(z)\big|\prec  \frac{1}{N\eta^{3/2}}\,.
\end{align*}
\item[$(ii)$] If $\mu_{\alpha}=\mu_{\beta}$, then, for any fixed (small) $\varsigma>0$, 
\begin{align*}
&\big|m_H(z)-m_{A\boxplus B}(z)\big|\prec \frac{1}{|z-1|}\frac{1}{N\eta^{3/2}} \,,
\end{align*}
uniformly on $\widetilde{\mathcal{S}}_{\mathcal{I}}^\varsigma(\eta_\mathrm{m},\eta_{\mathrm{M}})$.
\end{enumerate}
\end{proposition}

\begin{proof}[Proof of Proposition \ref{pro.a.1}] Note that, in the proof of Theorem~\ref{thm041801}, the only place where we use the assumption that at least one of $\mu_{\alpha}$ and $\mu_{\beta}$ is supported at more than two points is Lemma~\ref{lemma linear stability}; in particular in~\eqref{le bound on determinant}.  Hence, it suffices to mimic the proof of Theorem ~\ref{thm041801} with Lemma \ref{lemma linear stability} replaced by Proposition \ref{stability for two point mass}. Then the proof of the case  $\mu_\alpha\neq \mu_\beta$ is exactly the same as that of Theorem ~\ref{thm041801}. It suffices to discuss the case  $\mu_\alpha= \mu_\beta$ below.

Analogously to Corollary~\ref{cor081501}, with the aid of~\eqref{081901} and~\eqref{difference between subordination functions}, we show that
\begin{align}\label{081902}
 \Gamma_{\mu_A,\mu_B}(\omega_A(z),\omega_B(z))\lesssim\frac{1}{|z-1|}\,,\qquad\qquad z\in\mathcal{S}_{\mathcal{I}}^\varsigma(0,\eta_{\mathrm{M}})\,.
\end{align}
Then, we use a continuity argument, based on Lemma~\ref{le lemma large eta} and Proposition~\ref{le proposition perturbation of system} with $S$ replaced by $\frac{S}{|z-1|}$ therein, to deduce from~\eqref{071517} that $|\omega_i^c(z)-\omega_i(z)|\prec \|r^c(z)\|/|z-1|$, $i=A,B$, on $\widetilde{\mathcal{S}}_{\mathcal{I}}^\varsigma(\eta_{\mathrm{m}},\eta_{\mathrm{M}})$. The remaining parts of the proof are the same as in Theorem~\ref{thm041801}. This completes the proof of part $(ii)$ of Proposition~\ref{pro.a.1}.\qedhere

\end{proof}

\end{document}